\documentclass[11pt]{amsart}

\usepackage[margin=1.05in]{geometry}
\usepackage{amsmath,amssymb,amsthm,mathtools,mathrsfs}
\usepackage[T1]{fontenc}
\usepackage[utf8]{inputenc}
\usepackage{lmodern}
\usepackage{enumitem}
\usepackage{microtype}
\usepackage{hyperref}
\usepackage{xcolor}

\hypersetup{colorlinks=true,linkcolor=blue,citecolor=blue,urlcolor=blue}

\newtheorem{theorem}{Theorem}[section]
\newtheorem{proposition}[theorem]{Proposition}
\newtheorem{lemma}[theorem]{Lemma}
\newtheorem{corollary}[theorem]{Corollary}
\theoremstyle{remark}
\newtheorem{remark}[theorem]{Remark}

\newcommand{\C}{\mathbb C}
\newcommand{\om}{\Omega}
\newcommand{\Ric}{\operatorname{Ric}}
\newcommand{\Scal}{\operatorname{Scal}}
\newcommand{\tr}{\operatorname{tr}}
\newcommand{\ddbar}{\sqrt{-1}\,\partial\bar\partial}
\newcommand{\cB}{\mathcal B}
\newcommand{\gB}{g_B}
\newcommand{\Cn}{C_n}
\newcommand{\Ophg}{\mathcal O_{\mathrm{phg}}}
\newcommand{\Aphg}{\mathcal A_{\mathrm{phg}}}
\newcommand{\dotC}{\dot C^\infty}
\newcommand{\VTheta}{\mathcal V_\Theta}

\title[Extremal and cscK Bergman Metrics]{Rigidity of Extremal and cscK Bergman Metrics on Pseudoconvex Domains}
\author{Peter Ebenfelt}
\address{Department of Mathematics, University of California San Diego,
9500 Gilman Drive, La Jolla, CA 92093, USA}
\email{pebenfelt@ucsd.edu}

\author{Soumya Ganguly}
\address{Department of Mathematics, Rutgers University--New Brunswick,
110 Frelinghuysen Road, Piscataway, NJ 08854, USA}
\email{soumya.ganguly@rutgers.edu}

\thanks{The first author was supported in part by the NSF grant DMS-2453134.}

\date{}

\begin{document}
\begin{abstract}
Let $\Omega\subset\mathbb C^n$, $n\ge2$, be a bounded connected
pseudoconvex domain whose boundary contains a smooth strongly
pseudoconvex point. We prove that if the Bergman metric of $\Omega$ is
extremal, then its scalar curvature is identically $-n$, and that constant
scalar curvature forces the Bergman metric to be K\"ahler--Einstein.
Consequently, extremality, constant scalar curvature, and the
K\"ahler--Einstein condition are equivalent in this setting. A key analytic
ingredient is a local unique-continuation theorem for the Bergman
Laplacian at an ACH boundary, proved using  H\"ormander's Carleman
estimate. We also obtain an extension for possibly
unbounded pseudoconvex domains: if the Bergman metric is cscK or
extremal near such a boundary point, then it is well-defined and
K\"ahler--Einstein throughout the locus where $K_\Omega(z,z)>0$.
\end{abstract}

\maketitle

\section{Introduction}\label{sec:introduction}

Let $\om$ be a bounded pseudoconvex domain in $\C^n$; throughout this paper, we shall assume that $n\geq 2$. Associated with $\om$ is its Bergman metric, a central biholomorphically invariant K\"ahler metric with remarkable rigidity properties. A significant body of work has investigated how curvature properties of the Bergman metric constrain the underlying domain. A prominent example is the Cheng-Yau Conjecture, which asserts that a smoothly bounded pseudoconvex domain whose Bergman metric is K\"ahler--Einstein must be biholomorphic to the unit ball. The Cheng-Yau Conjecture is not the focus of this paper, although substantial progress has been made recently; see the works \cite{fu_wong_chengconjecture2d,NS06,HuangXiaoChengallD,HHL2026,SavaleXiao2025, HuangLi2023KEEisBergStein, EXX2022ChengSteinabelian, GangulySinha2026ChengStein2d, gangulysahi2026classificationsteinspacesbergmaneinstein} and references therein. Our focus instead is on rigidity among curvature conditions for the Bergman metric itself. We consider three successively weaker conditions on a K\"ahler metric: being K\"ahler--Einstein, having constant scalar curvature (cscK), and being extremal in the sense of Calabi. Recall that a K\"ahler metric is extremal if the $(1,0)$-gradient of its scalar curvature is holomorphic. Thus, we have
\[
\text{K\"ahler--Einstein}\implies\text{cscK} \implies \text{extremal},
\]
and both implications are strict for general K\"ahler metrics. Our main result shows that, for the Bergman metric, these three conditions are equivalent, provided that the boundary of $\om$ contains a smooth strongly pseudoconvex point. A notable previous result in this direction is due to Sha, who proved that a cscK Bergman metric on a smoothly bounded strongly
pseudoconvex domain is K\"ahler--Einstein \cite[Proposition~5.8]{Sha-cscK2025}.

%C-Y history
%In the case where $\om$ is strongly pseudoconvex, this conjecture has been established by Fu--Wong and Nemirovski--Shafikov \cite{fu_wong_chengconjecture2d,NS06} in the case $n=2$ and by Huang--Xiao \cite{HuangXiaoChengallD} for $n\geq 3$. More recently, Hsiao--Huang--Li \cite{HHL2026} proved the Cheng--Yau Conjecture for bounded pseudoconvex domains with real-analytic boundary; we also mention the work of Savale--Xiao \cite{SavaleXiao2025} in $\C^2$

%In our main result, we shall assume, in addition to $\om\subset \C^n$ being a bounded pseudoconvex domain, that the boundary $\partial\om$ is \(C^\infty\)-smooth and strongly pseudoconvex locally near one boundary point and prove that, for the Bergman metric of $\om$, the implications above are equivalences.

\begin{theorem}\label{cor:equivalence}
Let $\om\subset\C^n$, $n\ge2$, be a bounded connected pseudoconvex
domain whose boundary contains a smooth strongly pseudoconvex point. Then the following are equivalent:
\begin{enumerate}[label=\textup{(\roman*)}]
\item $\gB$ is extremal;
\item $\gB$ has constant scalar curvature;
\item $\gB$ is K\"ahler--Einstein.
\end{enumerate}
\end{theorem}

To establish Theorem \ref{cor:equivalence} we shall prove the implications (i) $\implies$  (ii) and (ii) $\implies$ (iii). This will be achieved in two separate steps as the mechanisms behind the two implications are different. We first establish the following:

\begin{theorem}\label{thm:extremal-csck}
Let $\om\subset\C^n$, $n\ge2$, be a bounded connected pseudoconvex
domain whose boundary contains a smooth strongly pseudoconvex point.  If the Bergman metric $\gB$ is extremal, then
\[
        \Scal(\gB)\equiv -n
\]
on $\om$.  In particular, $\gB$ has constant scalar curvature $=-n$.
\end{theorem}

The second step is then to demonstrate:

\begin{theorem}\label{thm:csck-ke}
Let $\om\subset\C^n$, $n\ge2$, be a bounded connected pseudoconvex
domain whose boundary contains a smooth strongly pseudoconvex point. If the Bergman metric \(g_B\)
has constant scalar curvature, then
\[
        \Ric(\gB)=-\gB.
\]
In particular, the Bergman metric is K\"ahler--Einstein with Einstein constant
$=-1$.
\end{theorem}

\begin{remark}
While completing this paper, we became aware of the
independent work of Li--Liu \cite{LiLiu2026csc}, which was posted on September 21, 2026, and where they prove, in complex dimensions \(n\ge3\), that the Bergman metric of a possibly unbounded
pseudoconvex domain is K\"ahler--Einstein if it has constant scalar
curvature and the boundary contains a smooth strongly pseudoconvex point. Thus, for $n\geq 3$, Theorem \ref{thm:csck-ke} is contained in \cite{LiLiu2026csc}. See also Remark \ref{rem:LL} below.
\end{remark}

The local nature of the boundary hypothesis is an important feature in the results above.  The local strongly pseudoconvex boundary patch supplies sufficiently precise
asymptotic information to deduce the global curvature conclusions. In particular, no
global smoothness assumption on $\partial\om$ is required.

For the reader's convenience, we give here a very brief executive summary of the ideas of proof of Theorems \ref{thm:extremal-csck} and \ref{thm:csck-ke}.

For the implication in Theorem \ref{thm:extremal-csck}, Proposition~\ref{prop:local-boundary-asymptotics}
shows that $\nabla_{\gB}^{1,0}\Scal(\gB)$ tends to zero on the strongly
pseudoconvex patch. Extremality makes this vector field holomorphic.
Boundary uniqueness and the identity theorem force it to vanish, and
the boundary asymptotics then determine the constant scalar curvature
to be $-n$.

The proof of Theorem \ref{thm:csck-ke} is more involved. The cscK condition implies that the normalized Bergman invariant $u=\log(\cB/\Cn)$  is harmonic with respect to the Bergman metric. Using localization of the Bergman kernel \cite{HHL2026} together with the Fefferman expansion, we obtain precise asymptotics for $u$ at the strongly pseudoconvex boundary patch. An indicial analysis of the Bergman Laplacian then shows that $u$ vanishes to infinite order there. The final step is a local unique-continuation theorem at the ACH boundary: we construct a tangentially localized Carleman weight and apply H\"ormander's estimate to conclude that $u$ vanishes on a one-sided collar. Real-analytic continuation then gives $u\equiv0$ on $\Omega$, which is equivalent to
\[
        \Ric(\gB)=-\gB.
\]
The passage to infinite-order vanishing requires separate arguments in complex dimension two and in higher dimensions; the latter uses a Monge--Amp\`ere rigidity argument on the localized strongly pseudoconvex domain.

The preceding results are stated for bounded domains, where the Bergman
kernel and Bergman metric are automatically nondegenerate throughout the
domain.  The argument is, however, substantially local near the strongly
pseudoconvex boundary patch.  Combining the localization theorem of
Hsiao--Huang--Li \cite{HHL2026} with the local Bergman nondegeneracy
result from our joint work with Sha and Xiao
\cite{EGSXscalarcurv2026}, one obtains the following result for possibly unbounded pseudoconvex
domains.  Note that in this setting the Bergman kernel may vanish at
interior points, and even when $K_\om(z,z)>0$ the Hermitian form given by $\partial\bar\partial\log K_\om(z,z)$ may fail to be positive definite. In the setting of our result, however, Proposition \ref{prop:unbounded-finite-type-nondegeneracy} (proved in \cite{EGSXscalarcurv2026}) below shows that the Bergman metric is well-defined on an open dense subset $\Omega^*\subset\Omega$, containing a one-sided neighborhood of every smooth strongly pseudoconvex boundary point. The result in the possibly unbounded case can be formulated as follows:

\begin{theorem}
\label{thm:unbounded-extension}
Let $\om\subset\C^n$, $n\ge2$, be a connected, possibly unbounded,
pseudoconvex domain whose boundary contains a smooth strongly pseudoconvex point $p_0$.
If the Bergman metric is extremal on a one-sided neighborhood of $p_0$, then it has constant scalar curvature $=-n$, and it is also K\"ahler--Einstein with Einstein constant $=-1$ on its domain of existence $\om^*$. Moreover, in this case, $\om^*$ equals $\om_K:=\{z\in\om:K_\om(z,z)>0\}$, where $K_\om(z,w)$ denotes the Bergman kernel of $\om$.
\end{theorem}

The global continuation in Theorem~\ref{thm:unbounded-extension} requires
a small additional observation.  One cannot a priori continue the Bergman
invariant across the zero set $Z_K$ of the Bergman kernel on the diagonal.  Instead, after the
local argument gives $\cB=\Cn$ on a boundary collar, we clear the
denominator and obtain a real-analytic identity on all of $\om$.  This
identity propagates across $Z_K$ and, on the positive-kernel locus,
forces the Bergman Hessian to be positive definite. The continuation takes place on $\om$ itself, including $Z_K$.

%For the extremal implication, Proposition~\ref{prop:local-boundary-asymptotics}
%shows that $\nabla_{\gB}^{1,0}\Scal(\gB)$ tends to zero on the strongly
%pseudoconvex patch. Extremality makes this vector field holomorphic.
%Boundary uniqueness and the identity theorem force it to vanish, and
%the boundary asymptotics then determine the constant scalar curvature
%to be $-n$.
%
%For the cscK implication, $u=\log(\cB/\Cn)$ satisfies
%$\Delta_{\gB}u=0$.  The localization technique in \cite{HHL2026} supplies its boundary expansion.
%In dimension two, $u=\Ophg(r^4)$ is already beyond the larger
%indicial root $2$. For $n\ge3$, indicial analysis first gives
%$u=O(r^n)$; the resulting Monge--Amp\`ere estimate forces sphericality
%of the boundary patch and improves this to $O(r^{n+2})$.
%In either case, the indicial argument then gives infinite-order vanishing.
%
%The local ACH geometry
%\cite{EpsteinMelroseMendoza1991,GuillarmouSaBarreto2008} permits a
%unique-continuation argument from this open boundary patch. Using the
%convexity mechanism of Vasy--Wunsch \cite{VasyWunsch2005}, we construct
%a tangentially localized weight and verify H\"ormander's quantitative
%Carleman condition. Proposition~\ref{prop:uci} gives $u=0$ on a
%nonempty collar, and real-analytic continuation completes the proof.

The paper is organized as follows.  Section~\ref{sec:prelim} fixes the
Bergman and curvature conventions and develops the common local boundary
asymptotics used in both the extremal and cscK arguments.
Section~\ref{sec:extremal-proof} proves
Theorem~\ref{thm:extremal-csck}.  Section~\ref{sec:csck-boundary}
reduces the constant-scalar-curvature condition to a harmonic equation and
records the localization estimates needed below.
Section~\ref{sec:theta-indicial} develops the local ACH geometry and proves
infinite-order boundary vanishing by indicial analysis. Section~\ref{sec:unique-continuation} proves the
local unique-continuation result and completes the proof of
Theorem~\ref{thm:csck-ke}.  Section~\ref{sec:unbounded-extension}
records the local Bergman nondegeneracy result needed for possibly
unbounded domains and proves
Theorem~\ref{thm:unbounded-extension}.

\begin{remark}\label{rem:LL}
As mentioned above, we became aware of the work of Li--Liu \cite{LiLiu2026csc} while completing the current paper. For $n\geq3$, their proof first obtains boundary sphericity and then uses a Biquard--Herzlich half-ball unique-continuation argument. The approach in this paper instead uses indicial analysis and a local ACH
unique-continuation theorem based on H\"ormander's Carleman estimate. Our approach here provides a unique-continuation result that applies at general strongly pseudoconvex points. In particular, our argument applies in complex dimension \(2\) without an
a priori sphericity assumption as in \cite{LiLiu2026csc}.
\end{remark}

\subsection*{Acknowledgments.}
The authors are grateful to Ming Xiao and Zehao Sha for many helpful
discussions and for allowing us to repeat some calculations from our joint work. The second author is grateful to Zehao Sha for
introducing him to extremal K\"ahler metrics, and to Achinta Kumar Nandi
for introducing him to ACH metrics and related techniques.

\section{Preliminaries and notation}\label{sec:prelim}

\subsection{Conventions on the Bergman metric}

Throughout, if \(X\) is a smooth manifold with boundary,
\(C^\infty(X)\) denotes the space of functions smooth up to
\(\partial X\).  In a local boundary calculation, \(f\in C^\infty\) (and similarly \(r^mC^\infty\)) means
that \(f\) extends smoothly to the closed one-sided collar under
consideration, after shrinking the collar if necessary.  On a
manifold without boundary, \(C^\infty(X)\) has its usual meaning.

Until Section~\ref{sec:unbounded-extension}, let
$\om\subset\C^n$, $n\ge2$, be a bounded connected pseudoconvex domain
whose boundary contains a smooth strongly pseudoconvex point. Let $K_\om(z,w)$ be the Bergman kernel on $\om$ with respect to Euclidean volume $dV$. We omit the subscript \(\om\) when we refer to quantities on \(\om\). On the diagonal, we write
\[
 K=K_\om(z,z),\qquad
 (\gB)_{\alpha\bar\beta}
 =\partial_\alpha\partial_{\bar\beta}\log K,
 \qquad
 G=\det\bigl((\gB)_{\alpha\bar\beta}\bigr),
 \qquad
 \cB=\frac{G}{K}.
\]
The K\"ahler form of $\gB$ is $\omega_B=\ddbar\log K$.  The Ricci curvature and the scalar curvature are defined by
\[
 \Ric(\gB)=-\ddbar\log G,
 \qquad
 \Scal(\gB)=\tr_{\gB}\Ric(\gB).
\]
Denote by \(\partial_\alpha=\partial/\partial z_\alpha\) and \(\partial_{\bar\beta}=\partial/\partial\bar z_\beta\). We use the complex Laplacian \(\Delta_{\gB} f=(\gB)^{\alpha\bar\beta}f_{\alpha\bar\beta}\), where repeated Greek indices are summed from \(1\) to \(n\). For a smooth real-valued function $f$, Fefferman's Monge--Amp\`ere determinant is
\begin{equation}\label{eq:J-definition}
 J(f)=(-1)^n\det
 \begin{pmatrix}
 f&f_{\bar\beta}\\
 f_\alpha&f_{\alpha\bar\beta}
 \end{pmatrix}.
\end{equation}
We use throughout the functions
\begin{equation}\label{eq:uv-global-definitions}
 v:=K^{-1/(n+1)},
 \qquad
 u:=\log\frac{\cB}{\Cn},\qquad \text{where}\quad
 \Cn=\frac{(n+1)^n\pi^n}{n!}.
\end{equation}
If a subscript \(D\) is displayed, the corresponding object is formed from
the Bergman kernel of the localized strongly pseudoconvex domain \(D\)
introduced below.

\subsection{Localization and adapted frames}
\label{subsec:localized-patch}

Let $M$ be a nonempty smooth strongly pseudoconvex open subset of
$\partial\om$. Fix $p\in M$, and choose a neighborhood $\widetilde U$ of
$p$ such that
\[
        \partial\om\cap\widetilde U\subset M.
\]
By the localization theorem of Hsiao--Huang--Li
\cite[Theorem~1.3]{HHL2026}, there exist a neighborhood
$U\Subset\widetilde U$ of $p$ and a bounded, smoothly bounded, strongly
pseudoconvex domain $D\subset\om\cap\widetilde U$ such that
\[
        U\cap\om=U\cap D
\]
and
\begin{equation}\label{eq:localized-kernel-difference}
        K_\om(z,w)-K_D(z,w)
        \in C^\infty\bigl((U\cap\overline D)\times(U\cap\overline D)\bigr).
\end{equation}
Choose a relatively compact open subpatch
\[
        p\in M_0\Subset M\cap U.
\]
All boundary estimates are understood locally and uniformly over a fixed
relatively compact subpatch. We shrink the patch and collar when needed. We choose a positive Fefferman defining function $r$ for $D$ on the common collar. Thus
\[
 \om\cap U=\{r>0\},\qquad
 \partial\om\cap U=\{r=0\},\qquad
 dr\ne0\quad\text{on }\partial\om\cap U.
\]
The same equations describe $D$ on this collar. Fefferman's expansion for $K_D$, together with \eqref{eq:localized-kernel-difference}, gives
\begin{equation}\label{eq:local-fefferman}
 K=\frac{\phi}{r^{n+1}}+\psi\log r,
\end{equation}
where $\phi$ and $\psi$ extend $C^\infty$-smoothly up to
$r=0$ on the fixed collar, and $\phi|_{r=0}>0$, see \cite{Fefferman74}. The defining function is normalized so that
\begin{equation}\label{eq:feff-r}
 J(r)=1+O(r^{n+1});
\end{equation}
see \cite{Fefferman76} and \cite[Section~2.1]{BergmanlogflatobsflatEX24}.

Relabeling coordinates if necessary, assume $r_n\ne0$ on $U$, and
define
\[
\begin{aligned}
L_1&=r_n\left(\frac{\partial}{\partial z_1}
       -\frac{r_1}{r_n}\frac{\partial}{\partial z_n}\right),\\
L_j&=\frac{\partial}{\partial z_j}
       -\frac{r_j}{r_n}\frac{\partial}{\partial z_n},
       \qquad 2\le j\le n-1,\\
L_n&=\frac1{r_n}\frac{\partial}{\partial z_n}.
\end{aligned}
\]
Thus $L_jr=0$ for $j<n$ and $L_nr=1$.  The transition matrix from the Euclidean holomorphic frame to $L_1,\ldots,L_n$ has determinant one. We use Greek indices to refer to Euclidean coordinates and Latin indices to the adapted complex frame. Thus
\[
 f_\alpha=\frac{\partial f}{\partial z_\alpha},\qquad
 f_{\alpha\bar\beta}
 =\frac{\partial^2f}{\partial z_\alpha\partial\bar z_\beta},
 \qquad
 f_i=L_if,\qquad
 f_{i\bar j}=\partial\bar\partial f(L_i,\overline{L_j}).
\]
The associated real adapted frame is
\[
 T_{2j-1}=\operatorname{Im}L_j,\quad
 T_{2j}=\operatorname{Re}L_j\quad(1\le j<n),\qquad
 T_{2n-1}=\operatorname{Im}L_n,\quad
 \partial_r=\operatorname{Re}L_n.
\]
The fields $T_1,\ldots,T_{2n-1}$ are tangent to the level sets of $r$, and $\partial_r r=1$.  We write
\[
 \mathcal T_{\mathrm{tan}}=\{T_1,\ldots,T_{2n-1}\}
\]
for the real adapted fields tangent to the level sets of $r$.

\subsection{Vanishing orders and polyhomogeneous functions}

For the arguments below, we use the following boundary-order and polyhomogeneous notation. We write $h\in O(r^k,\log r)$ if $h$ is smooth for $r>0$ and, on a sufficiently small one-sided neighborhood,
\[
 |h|\le Cr^k(1+|\log r|)^q
\]
for some constants $C>0$ and $q\in\mathbb N_0$.  The notation $O(r^k)$ means that the logarithmic factor is absent. Tangential derivatives are products of fields in $\mathcal T_{\mathrm{tan}}$, while the number of occurrences of $\partial_r$ in an adapted differential monomial will be called its normal order.

Let \(X_0=[0,\varepsilon)_r\times M_0\). A function \(h\) is polyhomogeneous on \(X_0\) if it has an expansion
\begin{equation}\label{eq:polyhomogeneous-definition}
 h\sim
 \sum_{(z,q)\in E}h_{zq}(y)r^z(\log r)^q,
 \qquad h_{zq}\in C^\infty(M_0),
\end{equation}
where \(E\subset\mathbb R\times\mathbb N_0\) is bounded below and contains only finitely many pairs with \(z<N\) for each \(N\).  The expansion is understood with the following differentiated remainder condition: for every real \(N\) and every adapted differential monomial \(Q\) of normal order
\(\nu\),
\[
 Q\left(
 h-\sum_{\substack{(z,q)\in E\\z<N}}
 h_{zq}(y)r^z(\log r)^q
 \right)
 \in O(r^{N-\nu},\log r).
\]
For \(m\in\mathbb R\), let \(\Aphg^m(X_0)\) be the class of such functions whose index set satisfies \(z\ge m\).  We set \(\Aphg(X_0)=\bigcup_m\Aphg^m(X_0)\).  In an identity, \(\Ophg(r^m)\) denotes an unspecified element of \(\Aphg^m(X_0)\), and
\[
 \Ophg(r^\infty):=\bigcap_{m\ge0}\Aphg^m(X_0).
\]
The same notation is used with any other smooth positive boundary defining function in place of $r$.  In particular,
\begin{equation}\label{eq:phg-calculus-rules}
 \partial_r\Ophg(r^m)=\Ophg(r^{m-1}),\qquad
 T\Ophg(r^m)=\Ophg(r^m)\quad(T\in\mathcal T_{\mathrm{tan}}),
\end{equation}
and products have the sum of the two orders. We shall also use the standard closure properties of the polyhomogeneous
calculus: smooth compositions preserve polyhomogeneity whenever the
composition is defined up to the boundary; in particular, if $F$ is
polyhomogeneous with positive smooth boundary value, then $F^\gamma$ and
$\log F$ are polyhomogeneous for every $\gamma\in\mathbb R$.  Likewise, a
matrix-valued polyhomogeneous function with invertible boundary value has a
polyhomogeneous inverse. Logarithmic factors imply the useful pointwise inclusion
\begin{equation}\label{eq:phg-to-big-O}
 \Ophg(r^m)\subset O(r^{m-\delta})
 \qquad\text{for every }\delta>0.
\end{equation}

\subsection{Basic Monge--Amp\`ere identities and variation formulas}
\label{subsec:MA-identities}

For every positive smooth function \(q\), the determinant
\eqref{eq:J-definition} satisfies
\begin{equation}\label{eq:J-alt}
        J(q)
        =
        q^{n+1}
        \det\bigl(
        \partial_\alpha\partial_{\bar\beta}(-\log q)
        \bigr).
\end{equation}
Indeed, this follows directly by taking the Schur complement in
\eqref{eq:J-definition}.

For
\[
        v=K^{-1/(n+1)}
\]
we therefore have
\[
        K=v^{-(n+1)},
        \qquad
        \gB=(n+1)\partial\bar\partial(-\log v),
\]
and hence
\begin{equation}\label{eq:B-J-identity}
        \cB=(n+1)^nJ(v).
\end{equation}
Consequently,
\begin{equation}\label{eq:u-J-identity}
        u
        =
        \log\left(\frac{n!}{\pi^n}J(v)\right).
\end{equation}

Since \(G=\cB K\),
\[
        \Ric(\gB)
        =
        -\ddbar\log(\cB K)
        =
        -\ddbar\log\cB-\gB.
\]
Taking the trace with respect to \(\gB\), and recalling that
\(u=\log(\cB/\Cn)\), gives
\begin{equation}\label{eq:scalar-u-identity}
        \Scal(\gB)+n
        =
        -\Delta_{\gB}u.
\end{equation}

We shall also use the following finite perturbation formulas for
Fefferman's Monge--Amp\`ere determinant, combining Lemmas~2.3--2.4 of
\cite{BergmanlogflatobsflatEX24}.

\begin{lemma}[Monge--Amp\`ere variation formulas]
\label{lem:MA-variation}
Let \(\chi,\eta\) be \(C^2\)-smooth near the boundary, with
\(\chi=0\) on the boundary and \(\chi>0\) on the interior side.
Then, for \(s\ge2\),
\begin{equation}\label{eq:J-power-variation}
        J(\chi+\eta\chi^s)
        =
        \left(
        1+s(n+2-s)\eta\chi^{s-1}
        \right)J(\chi)
        +O(\chi^s),
\end{equation}
and, for \(s\ge3\),
\begin{equation}\label{eq:J-log-variation}
\begin{aligned}
        J(\chi+\eta\chi^s\log\chi)
        &=
        \Bigl(
        1+s(n+2-s)\eta\chi^{s-1}\log\chi \\
        &\hspace{2cm}
        +(n+2-2s)\eta\chi^{s-1}
        \Bigr)J(\chi)
        +O(\chi^s\log\chi).
\end{aligned}
\end{equation}
\end{lemma}

\begin{remark}\label{rem:MA-variation-dim2}
Although the variation formulas above are used in
\cite{BergmanlogflatobsflatEX24} inside the proof of a proposition whose
geometric statement assumes \(n\ge3\), Lemmas~2.3--2.4 themselves have no
such dimension restriction.  In particular, when \(n=2\) and \(s=4\),
\[
        J(\chi+\eta\chi^4)
        =
        J(\chi)+O(\chi^4),
\]
whereas
\[
        J(\chi+\eta\chi^4\log\chi)
        =
        \bigl(1-4\eta\chi^3\bigr)J(\chi)
        +O(\chi^4\log\chi).
\]
\end{remark}

\subsection{Boundary asymptotics of the Bergman invariant and scalar curvature}
\label{subsec:common-local-asymptotics}

The following proposition is purely local and does not assume either
extremality or constant scalar curvature.  It is the common boundary input
used in both rigidity arguments.

\begin{proposition}
\label{prop:local-boundary-asymptotics}
On the strongly pseudoconvex collar fixed above, set
\[
        u:=\log\frac{\cB}{\Cn},
        \qquad
        f:=\Scal(\gB)+n.
\]
Then
\begin{equation}\label{eq:local-u-orders}
        u=\Ophg(r^2)\quad(n\ge3),
        \qquad
        u=\Ophg(r^4)\quad(n=2).
\end{equation}
Moreover, in the adapted frame,
\begin{equation}\label{eq:local-inverse-orders}
\begin{aligned}
        (\gB)^{a\bar b}&=\Ophg(r),&
        (\gB)^{a\bar n}&=\Ophg(r^2),\\
        (\gB)^{n\bar a}&=\Ophg(r^2),&
        (\gB)^{n\bar n}&=\Ophg(r^2),
        \qquad a,b<n.
\end{aligned}
\end{equation}
Consequently,
\begin{equation}\label{eq:local-scalar-orders}
        f=\Ophg(r^2)\quad(n\ge3),
        \qquad
        f=\Ophg(r^4)\quad(n=2),
\end{equation}
and, in every dimension \(n\ge2\),
\begin{equation}\label{eq:local-scalar-derivative-orders}
        \overline{L_n}f=\Ophg(r),
        \qquad
        \overline{L_a}f=\Ophg(r^2),
        \qquad a<n.
\end{equation}
\end{proposition}

\begin{proof}
The calculation in this proposition is part of the local scalar-curvature
analysis carried out in our joint work with Sha and Xiao
\cite{EGSXscalarcurv2026}. For smoothly bounded strictly pseudoconvex domains, Engli\v{s}
\cite[Theorem~1]{Englis08} obtained full boundary expansions for the
Bergman invariant, the Bergman metric and its inverse, and the Ricci and
scalar curvatures. After localization, these calculations apply on the
strongly pseudoconvex patch considered here. We require the sharper orders
in \eqref{eq:local-u-orders} and \eqref{eq:local-scalar-orders}, which are
derived below.

We first establish \eqref{eq:local-u-orders}.  Suppose \(n\ge3\).  Let
\[
        a=\frac{n!}{\pi^n},
        \qquad
        C=a^{-1/(n+1)}.
\]
By \cite[Proposition~2.1]{HuangXiaoChengallD}, which records the
boundary coefficient computation of Christoffers
\cite{Christoffers1981}, the smooth coefficient in Fefferman's expansion
for $D$ satisfies
\[
        \phi_D=a+r^2P_2,
        \qquad P_2\in C^\infty.
\]
The localization identity $K-K_D\in C^\infty$ replaces $\phi_D$ by
\[
        \phi=\phi_D+r^{n+1}h
\]
for some $h\in C^\infty$ and leaves the logarithmic coefficient unchanged.
Hence, after absorbing $r^{n-1}h$ into $P_2$, we still have
\[
        \phi=a+r^2P_2,
        \qquad P_2\in C^\infty.
\]
Hence
\[
        v=Cr+r^3a_3+\Ophg(r^4),
        \qquad a_3\in C^\infty.
\]
Here the logarithmic contribution to \(v\) begins only at order
\(r^{n+2}\), which is contained in \(\Ophg(r^4)\) for \(n\ge3\).

Put \(\chi=Cr\).  Writing \(r^3a_3=\eta\chi^3\), with
\(\eta\in C^\infty\), \eqref{eq:J-power-variation} with \(s=3\) gives
\[
        J(\chi+\eta\chi^3)
        =
        \bigl(1+3(n-1)\eta\chi^2\bigr)J(\chi)
        +O(\chi^3).
\]
All quantities on the left and all displayed terms on the right are
polyhomogeneous; consequently the pointwise \(O(\chi^3)\) remainder is
polyhomogeneous of order at least \(3\).  Since \(J(\chi)=C^{n+1}J(r)\),
we obtain
\[
        J(\chi+\eta\chi^3)
        =
        C^{n+1}J(r)+\Ophg(r^2).
\]
The remaining \(\Ophg(r^4)\) term in \(v\) has two-jet
\(\Ophg(r^2)\).  Since \(J\) is a polynomial in the two-jet of its
argument and the two-jet of \(\chi+\eta\chi^3\) is bounded up to the
boundary, this remainder changes \(J\) by only \(\Ophg(r^2)\).  Therefore
\[
        J(v)
        =
        C^{n+1}J(r)+\Ophg(r^2).
\]
Using \(J(r)=1+O(r^{n+1})\) from \eqref{eq:feff-r} and
\(C^{n+1}=\pi^n/n!\), we obtain
\[
        J(v)=\frac{\pi^n}{n!}+\Ophg(r^2).
\]
Equation \eqref{eq:B-J-identity} therefore gives
\[
        \frac{\cB}{\Cn}=1+\Ophg(r^2),
\]
and hence
\[
        u=\Ophg(r^2).
\]

Now suppose \(n=2\).  Write
\[
        c_0=\frac2{\pi^2},
        \qquad
        C=c_0^{-1/3}.
\]
The dimension-two coefficient calculation of Graham gives, in the
normalization used here,
\[
        \phi=c_0+O(r^3),
        \qquad
        \psi=-3c_0\eta_1+O(r),
\]
where \(\eta_1\) is normalized by a formal solution \(J(\rho)=1\)
with expansion
\[
        \rho\sim r\sum_{k\ge0}\eta_k(r^3\log r)^k,
        \qquad \eta_0=1+O(r^3).
\]
Only the boundary value of \(\eta_1\) is used; we choose a smooth
extension to the collar. See
\cite[p.~130 and equation~(3.1)]{GRA3_Graham_ScalarBoundaryInvariantsAndThe}
and \cite[Theorem~2.11 and equation~(4.11)]{GrahamhigherasympMongeAmpere}.
The smooth localization error changes \(\phi\) only by \(r^3C^\infty\)
and does not change \(\psi\).  Since \(\phi\) is smooth, Hadamard's lemma
gives a smooth function \(a\) such that
\[
        \phi=c_0+r^3a.
\]
Thus
\[
        K=r^{-3}\bigl(c_0+r^3a+r^3\psi\log r\bigr).
\]
Set
\[
        T=c_0^{-1}r^3(a+\psi\log r).
\]
Then
\[
\begin{aligned}
        v
        &=Cr(1+T)^{-1/3}\\
        &=Cr-\frac13Cc_0^{-1}r^4a
          -\frac13Cc_0^{-1}r^4\psi\log r
          +\Ophg(r^7).
\end{aligned}
\]
Hence
\[
        v
        =
        Cr+r^4\alpha+r^4\beta\log r+\Ophg(r^7),
        \qquad
        \alpha,\beta\in C^\infty,
\]
where
\[
        \beta
        =
        -\frac13Cc_0^{-1}\psi
        =
        C\eta_1+O(r).
\]

Put again \(\chi=Cr\).  Since
\[
        r^4\beta\log r
        =
        C^{-4}\beta\chi^4\log\chi
        -
        C^{-4}\beta(\log C)\chi^4,
\]
the constant \(\log C\) only changes the smooth \(\chi^4\)-coefficient.
Thus Remark~\ref{rem:MA-variation-dim2} shows that the smooth
\(r^4\)-perturbation has no contribution of order \(r^3\), whereas the
logarithmic perturbation contributes
\[
        -4C^2\beta r^3J(r).
\]

To combine the two finite perturbations, expand \(J(\chi+h)\) into its
part linear in the two-jet of \(h\) and terms containing at least two
entries from that two-jet, where
\[
        h=r^4\alpha+r^4\beta\log r.
\]
The linear part is the sum of the smooth and logarithmic contributions just
described.  Moreover,
\[
        h=\Ophg(r^4),\qquad
        dh=\Ophg(r^3),\qquad
        \partial\bar\partial h=\Ophg(r^2).
\]
Every determinant monomial containing at least two entries from the
perturbation two-jet is $\Ophg(r^5)$. For terms with exactly two
perturbed entries, two Hessian entries must be accompanied by the
background $(0,0)$ entry $\chi=O(r)$, giving $\Ophg(r^{1+2+2})$.
One perturbed Hessian entry and one perturbed first-derivative entry
give $\Ophg(r^{2+3})$, since the remaining background entry is bounded.
All other terms with exactly two perturbed entries have higher order.
Terms with three perturbed entries are $\Ophg(r^8)$. The \(\Ophg(r^7)\) remainder in \(v\) has
two-jet \(\Ophg(r^5)\) and hence changes \(J\) only by
\(\Ophg(r^5)\).  It follows that
\[
        J(v)
        =
        C^3J(r)-4C^2\beta r^3J(r)+\Ophg(r^4).
\]
Since $r$ is a Fefferman defining function, the quotient
\[
        q_r:=\frac{J(r)-1}{r^3}
\]
extends smoothly to the boundary.  Graham's equation~(4.11), specialized
to $n=2$, gives
\[
        q_r\big|_{r=0}=4\eta_1\big|_{r=0};
\]
see \cite[equation~(4.11)]{GrahamhigherasympMongeAmpere}.  Since
$\eta_1$ is smooth, Hadamard's lemma therefore yields
\[
        J(r)=1+4\eta_1r^3+\Ophg(r^4).
\]
Hence the order-$r^3$ terms cancel:
\[
\begin{aligned}
        J(v)
        &=
        C^3\bigl(1+4\eta_1r^3\bigr)
        -4C^2\bigl(C\eta_1+O(r)\bigr)r^3
        +\Ophg(r^4)\\
        &=
        C^3+\Ophg(r^4)
        =
        \frac{\pi^2}{2}+\Ophg(r^4).
\end{aligned}
\]
Since \(\cB=9J(v)\) and \(C_2=9\pi^2/2\), we obtain
\[
        \frac{\cB}{C_2}=1+\Ophg(r^4),
        \qquad
        u=\Ophg(r^4).
\]
This proves \eqref{eq:local-u-orders}.

We next prove the inverse-metric estimates.  Write
\[
        F:=r^{n+1}K
        =\phi+\psi r^{n+1}\log r.
\]
Then \(F\) is polyhomogeneous and has a positive smooth boundary value.
Since
\[
        \gB
        =
        (n+1)\partial\bar\partial(-\log r)
        +\partial\bar\partial\log F,
\]
and
\[
        \partial\bar\partial(-\log r)
        =
        \frac{\partial r\otimes\bar\partial r}{r^2}
        -\frac{\partial\bar\partial r}{r},
\]
the adapted-frame matrix of \(\gB\) has the block form
\[
        \gB=
        \begin{pmatrix}
        r^{-1}A+\Ophg(r^0)
        &
        r^{-1}B+\Ophg(r^0)
        \\[3pt]
        r^{-1}B^*+\Ophg(r^0)
        &
        (n+1)r^{-2}+r^{-1}c+\Ophg(r^0)
        \end{pmatrix},
\]
where \(A\) is a smooth positive definite
\((n-1)\times(n-1)\) matrix and \(B,c\) are smooth up to the boundary.
Factoring the preceding matrix as
\[
        \gB=D_rM_r,
        \qquad
        D_r=
        \begin{pmatrix}
        r^{-1}I_{n-1}&0\\
        0&r^{-2}
        \end{pmatrix},
\]
the matrix \(M_r\) is polyhomogeneous up to the boundary and has the
invertible boundary value
\[
        \left.M_r\right|_{r=0}=
        \begin{pmatrix}
        A|_{r=0}&B|_{r=0}\\
        0&n+1
        \end{pmatrix}.
\]
Hence \(M_r^{-1}=\Ophg(r^0)\), and
\[
        (\gB)^{-1}
        =
        M_r^{-1}
        \begin{pmatrix}
        rI_{n-1}&0\\
        0&r^2
        \end{pmatrix}.
\]
Using also the Hermitian symmetry of \((\gB)^{-1}\), we obtain
\[
        (\gB)^{a\bar b}=\Ophg(r),\qquad
        (\gB)^{a\bar n}=\Ophg(r^2),\qquad
        (\gB)^{n\bar a}=\Ophg(r^2),\qquad
        (\gB)^{n\bar n}=\Ophg(r^2),
\]
which proves \eqref{eq:local-inverse-orders}.

Finally, by \eqref{eq:scalar-u-identity},
\[
        f=-\Delta_{\gB}u.
\]
For later use, we record explicitly how the complex Hessian acts on the
polyhomogeneous scale.  Since the adapted frame is smooth and
nondegenerate up to \(r=0\), for each \(i,j\) we have
\[
        \partial\bar\partial w(L_i,\overline{L_j})
        =
        L_i\overline{L_j}w-\mathcal L_{i\bar j}w,
\]
where \(\mathcal L_{i\bar j}\) is a first-order differential operator
with smooth coefficients and no zero-order term.  Since every smooth
first-order vector field is a smooth linear combination of the real
adapted fields, if \(w=\Ophg(r^m)\), then
\[
        \mathcal L_{i\bar j}w=\Ophg(r^{m-1}).
\]
Together with the differentiation rules
\eqref{eq:phg-calculus-rules}, this gives
\[
\begin{aligned}
        w_{a\bar b}&=\Ophg(r^{m-1}),\\
        w_{a\bar n},\,w_{n\bar a}&=\Ophg(r^{m-1}),\\
        w_{n\bar n}&=\Ophg(r^{m-2}),
        \qquad a,b<n.
\end{aligned}
\]
Consequently, using \eqref{eq:local-inverse-orders},
\[
\begin{aligned}
        (\gB)^{a\bar b}w_{a\bar b}
        &=\Ophg(r)\Ophg(r^{m-1})
          =\Ophg(r^m),\\
        (\gB)^{a\bar n}w_{a\bar n}
        +(\gB)^{n\bar a}w_{n\bar a}
        &=\Ophg(r^{m+1})
          \subset\Ophg(r^m),\\
        (\gB)^{n\bar n}w_{n\bar n}
        &=\Ophg(r^2)\Ophg(r^{m-2})
          =\Ophg(r^m).
\end{aligned}
\]
Hence
\[
        \Delta_{\gB}\Ophg(r^m)=\Ophg(r^m).
\]
Applying this to \eqref{eq:local-u-orders} yields
\[
        f=\Ophg(r^2)\quad(n\ge3),
        \qquad
        f=\Ophg(r^4)\quad(n=2).
\]
Therefore
\[
\begin{array}{lll}
n\ge3:
&
\overline{L_n}f=\Ophg(r),
&
\overline{L_a}f=\Ophg(r^2),
\\[3pt]
n=2:
&
\overline{L_n}f=\Ophg(r^3),
&
\overline{L_a}f=\Ophg(r^4).
\end{array}
\]
In particular, in every dimension \(n\ge2\),
\[
        \overline{L_n}f=\Ophg(r),
        \qquad
        \overline{L_a}f=\Ophg(r^2),
        \qquad a<n,
\]
which proves \eqref{eq:local-scalar-derivative-orders}.
\end{proof}

\section{Extremal Bergman metrics}\label{sec:extremal-proof}

In this section, we give a proof of Theorem \ref{thm:extremal-csck}. Recall that $\gB$ is extremal if
\begin{equation}\label{eq:extremal}
        \mathcal X:=\nabla_{\gB}^{1,0}\Scal(\gB)
        =(\gB)^{\alpha\bar\beta}
          \partial_{\bar\beta}\Scal(\gB)
          \frac{\partial}{\partial z_\alpha}
\end{equation}
is a holomorphic \((1,0)\)-vector field.

\begin{lemma}\label{lem:X-decays}
Suppose that $\gB$ is extremal on the fixed strongly pseudoconvex
collar over $M_0$.  Then we have
\begin{equation}\label{eq:X-decays}
        \mathcal X
        =\Ophg(r^3)L_n
          +\sum_{a=1}^{n-1}\Ophg(r^3)L_a.
\end{equation}
In particular, the components of $\mathcal X$ in Euclidean coordinates are bounded and tend to zero at the boundary patch.
\end{lemma}

\begin{proof}
Set \(f=\Scal(\gB)+n\).  By
Proposition~\ref{prop:local-boundary-asymptotics},
\[
        \overline{L_n}f=\Ophg(r),
        \qquad
        \overline{L_a}f=\Ophg(r^2),
        \qquad a<n,
\]
and
\[
\begin{aligned}
        (\gB)^{a\bar b}&=\Ophg(r),&
        (\gB)^{a\bar n}&=\Ophg(r^2),\\
        (\gB)^{n\bar a}&=\Ophg(r^2),&
        (\gB)^{n\bar n}&=\Ophg(r^2).
\end{aligned}
\]
In the adapted frame,
\[
        \mathcal X
        =
        (\gB)^{A\bar B}\overline{L_B}f\,L_A.
\]
Thus the coefficient of \(L_n\) satisfies
\[
\begin{aligned}
        (\gB)^{n\bar B}\overline{L_B}f
        &=
        (\gB)^{n\bar n}\overline{L_n}f
        +(\gB)^{n\bar b}\overline{L_b}f\\
        &=
        \Ophg(r^2)\Ophg(r)
        +\Ophg(r^2)\Ophg(r^2)
        =
        \Ophg(r^3),
\end{aligned}
\]
while, for \(a<n\),
\[
\begin{aligned}
        (\gB)^{a\bar B}\overline{L_B}f
        &=
        (\gB)^{a\bar n}\overline{L_n}f
        +(\gB)^{a\bar b}\overline{L_b}f\\
        &=
        \Ophg(r^2)\Ophg(r)
        +\Ophg(r)\Ophg(r^2)
        =
        \Ophg(r^3).
\end{aligned}
\]
This proves \eqref{eq:X-decays}.  Since the adapted frame \(L_A\) is smooth
and nondegenerate up to the boundary, the Euclidean components of
\(\mathcal X\) have the same decay and hence tend to zero on the boundary
patch.
\end{proof}

\begin{lemma}\label{lem:boundary-uniqueness}
Let \(\mathcal D\subset\C^n\) be a domain with a \(C^1\)-smooth boundary patch
\(M\), and let \(U\) be a neighborhood of a point of \(M\).  Suppose that
\(h\in\mathcal O(\mathcal D\cap U)\) and
\[
        \lim_{\substack{z\to p\\z\in \mathcal D\cap U}}h(z)=0
        \qquad (p\in M\cap U).
\]
After shrinking \(U\), the function \(h\) vanishes on the component of
\(\mathcal D \cap U\) adjacent to \(M\).
\end{lemma}

\begin{proof}
After a complex-affine change of coordinates and a shrinking of $U$, write
$z=(z',z_n)$, where $z'\in\C^{n-1}$ and $z_n=x_n+iy_n$, so that
\[
        M=\{x_n=\sigma(z',y_n)\}
\]
with $\sigma\in C^1$, and the component of $\mathcal D\cap U$ adjacent to
$M$ lies, say, in $\{x_n<\sigma(z',y_n)\}$.

Fix $z'$ in a sufficiently small open set.  On a small rectangle $R$ in
the $z_n$-plane, the slice of $M$ is the $C^1$ arc
\[
        \Gamma_{z'}
        =\{x_n+iy_n:x_n=\sigma(z',y_n)\},
\]
which separates $R$ into two sides $R_-$ and $R_+$, with $R_-$ contained
in the $\mathcal D$-side.  Set
\[
        f_{z'}(\zeta)=h(z',\zeta),\qquad \zeta\in R_-,
\]
and define $F_{z'}$ on $R$ by $F_{z'}=f_{z'}$ on $R_-$ and
$F_{z'}=0$ on $\Gamma_{z'}\cup R_+$.  The boundary-limit hypothesis makes
$F_{z'}$ continuous on $R$.

We claim that $F_{z'}$ is holomorphic on $R$. Put
$\rho(\zeta)=x_n-\sigma(z',y_n)$, and let $\kappa_\epsilon$ be a
$C^1$ cutoff equal to zero for $|\rho|\le\epsilon$ and one for
$|\rho|\ge2\epsilon$, with $|d\kappa_\epsilon|\le C/\epsilon$.
For $\eta\in C_c^\infty(R)$, integration by parts away from the arc gives
\[
 \int_R\kappa_\epsilon F_{z'}\,\bar\partial\eta
 =-\int_R F_{z'}\eta\,\bar\partial\kappa_\epsilon.
\]
On the support of $\eta$, the strip $|\rho|<2\epsilon$ has area
$O(\epsilon)$, and the supremum of $|F_{z'}|$ on that strip tends
to zero. The right-hand side therefore tends to zero. Passing to
the limit shows that $\bar\partial F_{z'}=0$ in distributions.
Since $F_{z'}$ is continuous, it is holomorphic on $R$.

Since $F_{z'}$ vanishes on the nonempty open set $R_+$, the one-variable
identity theorem yields $F_{z'}\equiv0$ on $R$.  Hence $h(z',\cdot)$
vanishes on the adjacent part of the slice.  Allowing $z'$ to vary over a
small open set shows that $h$ vanishes on a nonempty open subset of the
component of $\mathcal D\cap U$ adjacent to $M$.  The
identity theorem then gives the asserted vanishing on that component.
\end{proof}

\begin{corollary}\label{cor:X-zero-collar}
Suppose that $\gB$ is extremal on the fixed strongly pseudoconvex
collar over $M_0$.  Then, after replacing $M_0$ by a smaller nonempty
relatively compact open patch and shrinking the collar,
\[
        \mathcal X\equiv0
\]
on that collar.
\end{corollary}

\begin{proof}
In local Euclidean coordinates, write
\[
        \mathcal X=\sum_{\alpha=1}^n X^\alpha
        \frac{\partial}{\partial z_\alpha}.
\]
Since $\gB$ is extremal, the vector field $\mathcal X$ is holomorphic, so
its Euclidean coefficients $X^\alpha$ are holomorphic.  By
Lemma~\ref{lem:X-decays}, every $X^\alpha$ tends to zero at the boundary
patch.  Lemma~\ref{lem:boundary-uniqueness}, applied componentwise, gives
$X^\alpha\equiv0$ on the adjacent one-sided collar.  Hence
$\mathcal X\equiv0$ there.
\end{proof}

\begin{proof}[Proof of Theorem~\ref{thm:extremal-csck}]
By Corollary~\ref{cor:X-zero-collar}, the holomorphic vector field
\[
        \mathcal X=\nabla_{\gB}^{1,0}\Scal(\gB)
\]
vanishes on a nonempty one-sided collar.  Its Euclidean coordinate
components are holomorphic on the connected domain $\om$, so the
holomorphic identity theorem gives
\[
        \mathcal X\equiv0
        \qquad\text{on }\om.
\]
Since the matrix $((\gB)^{\alpha\bar\beta})$ is invertible,
\eqref{eq:extremal} then implies
\[
        \bar\partial\Scal(\gB)=0
        \qquad\text{on }\om.
\]
The scalar curvature is real-valued, hence it is constant on $\om$.
Finally, Proposition~\ref{prop:local-boundary-asymptotics} gives
\[
        \Scal(\gB)+n\longrightarrow0
        \qquad\text{as }r\to0
\]
along the strongly pseudoconvex patch.  Therefore, the constant value is
$-n$, and
\[
        \Scal(\gB)\equiv-n
        \qquad\text{on }\om.\qedhere
\]
\end{proof}

\section{Constant scalar curvature: harmonic reduction and localization}
\label{sec:csck-boundary}

\subsection{The harmonic equation for the Bergman invariant}

\begin{lemma}\label{lem:scalar-pde}
Suppose that $\Scal(\gB)$ is constant on the fixed strongly
pseudoconvex collar over $M_0$.  Then this constant is $-n$, and
\begin{equation}\label{eq:harmonic-u}
        \Delta_{\gB}\log\frac{\cB}{\Cn}=0
\end{equation}
on that collar.
\end{lemma}

\begin{proof}
By Proposition~\ref{prop:local-boundary-asymptotics},
\[
        \Scal(\gB)(z)\longrightarrow-n
        \qquad\text{as }z\to p\in M_0.
\]
Hence the constant value of $\Scal(\gB)$ on the collar is $-n$.
Equation~\eqref{eq:scalar-u-identity} therefore gives
\[
        \Delta_{\gB}\log\frac{\cB}{\Cn}=0
\]
on the collar.
\end{proof}

\subsection{Kernel localization and the Monge--Amp\`ere determinant}

\begin{lemma}\label{lem:local-error}
Let \(D\) be the strongly pseudoconvex domain from
Subsection~\ref{subsec:localized-patch}, and write
\[
        K=K_D+\Phi,\qquad \Phi\in C^\infty,
\]
near the common boundary patch.  Put
\[
        v=K^{-1/(n+1)},\qquad v_D=K_D^{-1/(n+1)}.
\]
Then
\[
        v-v_D=r^{n+2}a+\Ophg(r^{2n+3})
\]
for some \(a\in C^\infty\).  Moreover,
\[
        J(v)-J(v_D)=O(r^{n+2}).
\]
\end{lemma}

\begin{proof}
Since \(K_D\sim cr^{-n-1}\) and \(\Phi\) is smooth,
\[
        \varepsilon:=\frac{\Phi}{K_D}
        =r^{n+1}e+\Ophg(r^{2n+2}),
        \qquad e\in C^\infty.
\]
Consequently,
\[
        v
        =v_D(1+\varepsilon)^{-1/(n+1)}
        =v_D-\frac1{n+1}v_D\varepsilon
          +\Ophg(r^{2n+3}).
\]
Since \(v_D=r b+\Ophg(r^{n+2})\), where \(b\) is smooth and positive at
the boundary, this gives
\begin{equation}\label{eq:localization-v-difference}
        v-v_D=r^{n+2}a+R,
        \qquad R=\Ophg(r^{2n+3}).
\end{equation}

From Fefferman's expansion for $D$,
\[
        \frac{v_D}{r}
        =
        \left(\phi_D+\psi_Dr^{n+1}\log r\right)^{-1/(n+1)}.
\]
Since \(n\ge2\), the function \(r^{n+1}\log r\), extended by zero at
\(r=0\), is \(C^2\) up to the boundary.  Thus \(v_D/r\) extends
positively with at least \(C^2\)-regularity, and we may write
\[
        r^{n+2}a=\eta v_D^{n+2},
        \qquad
        \eta=a\left(\frac{r}{v_D}\right)^{n+2}\in C^2.
\]
At the critical exponent \(s=n+2\), \eqref{eq:J-power-variation} gives
\[
        J(v_D+\eta v_D^{n+2})
        =
        J(v_D)+O(v_D^{n+2})
        =
        J(v_D)+O(r^{n+2}),
\]
because the coefficient \(s(n+2-s)\) vanishes.

It remains to control the remainder
\(R=\Ophg(r^{2n+3})\).  Its entries in the two-jet matrix satisfy
\[
        R=\Ophg(r^{2n+3}),\qquad
        dR=\Ophg(r^{2n+2}),\qquad
        \partial\bar\partial R=\Ophg(r^{2n+1}).
\]
The two-jet of \(v_D+\eta v_D^{n+2}\) is bounded up to the boundary.
Since \(J\) is a polynomial in the entries of the two-jet matrix, every term in
\[
        J(v_D+\eta v_D^{n+2}+R)
        -
        J(v_D+\eta v_D^{n+2})
\]
contains at least one entry from the two-jet of \(R\).  Hence this difference
is \(\Ophg(r^{2n+1})\).  Since \(n\ge2\), we have
\(2n+1>n+2\), and therefore
\[
        \Ophg(r^{2n+1})\subset O(r^{n+2}).
\]
Combining the two estimates gives
\[
        J(v)-J(v_D)=O(r^{n+2}),
\]
as claimed.
\end{proof}

\subsection{Sphericality and improved boundary decay in higher dimensions}
\label{subsec:localized-input}

The following lemma is the local form of the implication from asymptotic
Monge--Amp\`ere flatness to spherical boundary geometry that is needed below.

\begin{lemma}\label{lem:EX-local-use}
Let $n\ge3$. Suppose that on the collar over an open patch $M_1\subset\partial D$, we have
\[
 J\left(K_D^{-1/(n+1)}\right)=\frac{\pi^n}{n!}+o(r^2).
\]
Then, after replacing $M_1$ by any relatively compact subpatch,
\[
 J\left(K_D^{-1/(n+1)}\right)=\frac{\pi^n}{n!}+O(r^{n+2}).
\]
\end{lemma}

\begin{proof}
Write $v_D=K_D^{-1/(n+1)}$. By \cite[Proposition~2.5]{BergmanlogflatobsflatEX24}, every point of $M_1$ is CR umbilical.  Since $n\ge3$, the Chern--Moser curvature therefore vanishes on the open patch $M_1$, which is locally spherical.

Fix $M_2\Subset M_1$.  Let $\rho_D$ be the Cheng--Yau defining function, normalized by
\begin{equation}\label{eq:local-exact-CY}
 J(\rho_D)=1,\qquad \rho_D>0\text{ in }D,\qquad \rho_D=0\text{ on }\partial D.
\end{equation}
The Lee--Melrose expansion
\cite[Theorem~8.18]{LeeMelroseMongeAmpere} gives the full asymptotic
expansion of $\rho_D$ at the boundary.  Locally, however, such an
asymptotic solution is not completely determined by the boundary
hypersurface: Graham's Theorem~2.11 shows that one arbitrary smooth
boundary function remains, while the logarithmic coefficient functions
modulo the stated order are independent of this choice.  In particular,
the leading logarithmic coefficient restricted to the boundary is a local
CR invariant.  Since $M_1$ is locally spherical, this invariant vanishes
on $M_1$.  Graham's Proposition~2.16
\cite{GrahamhigherasympMongeAmpere} then implies that every logarithmic
coefficient is flat along $M_2$.  The Lee--Melrose expansion therefore has
no logarithmic singularities to any order along $M_2$, and its differentiated
remainder estimates imply that $\rho_D$ extends smoothly to $M_2$.  Its
leading term is a positive smooth multiple of a positive Fefferman defining
function, so $\rho_D$ is itself a positive defining function there.

Because $\rho_D$ is smooth up to $M_2$ and satisfies $J(\rho_D)=1$, it is itself a Fefferman defining function on this collar.  We may therefore write Fefferman's expansion relative to $\rho_D$ in the form
\[
 K_D=\frac{\phi_D}{\rho_D^{n+1}}+\psi_D\log \rho_D,
 \qquad
 \phi_D,\psi_D\in C^\infty.
\]
For a Fefferman defining function, $\phi_D$ modulo $O(\rho_D^{n+1})$ and $\psi_D$ modulo $O(\rho_D^\infty)$ are determined by the local CR geometry of the boundary.  Since $M_1$ is locally spherical, comparison with the unit sphere gives
\[
 \phi_D=\frac{n!}{\pi^n}+\rho_D^{n+1}a,
 \qquad
 \psi_D=O(\rho_D^\infty),
 \qquad
 a\in C^\infty;
\]
see, in particular, \cite[proof of Proposition~2.1]{BergmanlogflatobsflatEX24}. Since $\psi_D$ is flat, $\psi_D\log \rho_D$ extends smoothly and is  flat at $M_2$.  Hence, after absorbing this term into a smooth function $A$,
\begin{equation}\label{eq:local-spherical-Bergman-expansion}
 K_D=\frac{n!}{\pi^n}\rho_D^{-n-1}+A,
 \qquad A\in C^\infty.
\end{equation}

Write
\[
 c_0=\frac{n!}{\pi^n},
 \qquad
 C=c_0^{-1/(n+1)},
 \qquad
 \chi=C\rho_D.
\]
Then
\[
 J(\chi)=C^{n+1}J(\rho_D)=\frac{\pi^n}{n!}.
\]
Moreover, \eqref{eq:local-spherical-Bergman-expansion} gives
\[
\begin{aligned}
 v_D
 &=K_D^{-1/(n+1)}\\
 &=C\rho_D
 \left(1+c_0^{-1}A\,\rho_D^{n+1}\right)^{-1/(n+1)}
 =\chi+\eta\chi^{n+2}
\end{aligned}
\]
for some $\eta\in C^\infty$ on a sufficiently small collar.  Thus
\begin{equation}\label{eq:EX-local-v-expansion}
 v_D=\chi+\eta\chi^{n+2}.
\end{equation}
Applying \cite[Lemma~2.3]{BergmanlogflatobsflatEX24} with $s=n+2$, the coefficient of $\chi^{n+1}$ vanishes, and hence
\[
 J(v_D)=J(\chi)+O(\chi^{n+2})
 =\frac{\pi^n}{n!}+O(\rho_D^{n+2}).
\]
Finally, since $\rho_D/r$ extends smoothly and positively to $M_2$,
\[
 J(v_D)=\frac{\pi^n}{n!}+O(r^{n+2}),
\]
as required.
\end{proof}

\section{Local ACH geometry and indicial analysis}
\label{sec:theta-indicial}

In this section all ACH operators, equations, and Carleman estimates concern
\(\Delta_{\gB}\) on the fixed strongly pseudoconvex collar
\[
        X_0=[0,\varepsilon)_r\times M_0.
\]
We use the square-root compactification
\[
        x=r^{1/2},
        \qquad
        X_\Theta=[0,\sqrt\varepsilon)_x\times M_0.
\]
On \(X_\Theta\), the notation \(\Aphg^m\) and
\(\Ophg(x^m)\) is understood with \(x\) as the smooth boundary defining
function.

\subsection{The contact form and the local ACH coframe}

Let \(J_{\mathrm{st}}\) denote the standard complex structure on \(\C^n\).
For a real-valued function \(f\), write
\[
        d^cf
        =
        \frac12\sum_j
        \bigl(f_{x_j}\,dy_j-f_{y_j}\,dx_j\bigr),
        \qquad
        d^cf(X)=-\frac12df(J_{\mathrm{st}}X).
\]
Put
\[
        H=TM_0\cap J_{\mathrm{st}}TM_0,
        \qquad
        \theta
        =
        \left.
        \frac{\sqrt{-1}}2(\partial r-\bar\partial r)
        \right|_{M_0}.
\]
Then
\[
        \ker\theta=H.
\]
Since \(M_0\) is strongly pseudoconvex, the Levi metric
\[
        h_\theta(X,Y)=d\theta(X,J_{\mathrm{st}}Y),
        \qquad X,Y\in H,
\]
is positive definite.

Let \(T\) be the Reeb field,
\[
        \theta(T)=1,
        \qquad
        d\theta(T,\cdot)=0,
\]
and choose a local frame \(Y_1,\ldots,Y_{2n-2}\) of \(H\).
The local \(\Theta\)-vector fields are generated over
\(C^\infty(X_\Theta)\) by
\[
        x\partial_x,\qquad
        xY_1,\ldots,xY_{2n-2},\qquad
        x^2T.
\]
We denote their \(C^\infty(X_\Theta)\)-module by
\(\VTheta(X_\Theta)\).  Its dual coframe is
\[
        \frac{dx}{x},
        \qquad
        \frac{\eta^1}{x},\ldots,\frac{\eta^{2n-2}}{x},
        \qquad
        \frac{\theta}{x^2},
\]
where \(\eta^1,\ldots,\eta^{2n-2}\) is a local coframe on \(H\).
A \(\Theta\)-tensor is \(O_\Theta(x^m)\) if its coefficients in a smooth
\(\Theta\)-frame, together with all their \(\Theta\)-derivatives, are
\(O(x^m)\).  We also write
\[
        \dotC(X_\Theta)
        =
        \{f\in C^\infty(X_\Theta):
          f=O(x^N)\text{ for every }N\}.
\]

If \(\omega\) is a K\"ahler form, its associated real Riemannian
metric is
\[
g^{\mathrm R}(X,Y)=\omega(X,J_{\mathrm{st}}Y).
\]
Thus, when
\[
\omega=\sqrt{-1}g_{\alpha\bar\beta}\,
dz_\alpha\wedge d\bar z_\beta,
\]
we have
\[
g^{\mathrm R}
=
2\operatorname{Re}
\bigl(g_{\alpha\bar\beta}\,
dz_\alpha\otimes d\bar z_\beta\bigr).
\]
More generally, if \(q\) is a real \((1,1)\)-form,
\(q^{\mathrm R}\) denotes the associated symmetric real two-tensor. In
particular, \(g_B^{\mathrm R}\) is associated with \(\omega_B\).  For a
Riemannian metric \(g\) and a covector \(\xi\), the notation
\(\xi^\sharp\) denotes the \(g\)-dual vector.

Set
\[
        \omega_0=(n+1)\ddbar(-\log r),
\]
and let \(g_0^{\mathrm R}\) be its real Riemannian metric.  The local
strict pseudoconvex normal form in the preceding coframe shows that
\(g_0^{\mathrm R}\) is a smooth positive-definite \(\Theta\)-metric whose
boundary value is the complex-hyperbolic model.  Hence, in the adapted
\(\Theta\)-coframe,
\begin{equation}\label{eq:local-ACH-leading}
g_0^{\mathrm R}
=
c_{\mathrm{rad}}\frac{dr^2}{r^2}
+
c_{\mathrm{Reeb}}\frac{\theta^2}{r^2}
+
c_{\mathrm{hor}}\frac{h_\theta}{r}
+
O_\Theta(x),
\qquad
c_{\mathrm{rad}},c_{\mathrm{Reeb}},c_{\mathrm{hor}}>0,
\end{equation}
where \(O_\Theta(x)\) denotes a smooth \(\Theta\)-tensor whose coefficients
vanish at least to first order in \(x\), with the same property after every
fixed number of \(\Theta\)-derivatives.

The localized Fefferman expansion is
\[
        K
        =r^{-(n+1)}
        \bigl(\phi+\psi r^{n+1}\log r\bigr),
        \qquad
        \phi,\psi\in C^\infty(X_0),\quad \phi|_{M_0}>0.
\]
Therefore
\[
        g_B^{\mathrm R}-g_0^{\mathrm R}
        =\bigl(\ddbar\log \phi\bigr)^{\mathrm R}
        +\left(\ddbar
        \log\left(1+\frac{\psi}{\phi}
        r^{n+1}\log r\right)\right)^{\mathrm R}.
\]
Both terms have coefficients vanishing to positive order in the
\(\Theta\)-coframe.  Thus \(g_B^{\mathrm R}\) has the same boundary
value as \(g_0^{\mathrm R}\), and is a polyhomogeneous ACH \(\Theta\)-metric
on the collar.  In the variable \(x=r^{1/2}\), the horizontal and Reeb
length scales are respectively \(x^{-1}\) and \(x^{-2}\).

\subsection{Normal and indicial operators}

The normal operator freezes a \(\Theta\)-operator at a boundary point.  Locally, if
\[
        P=\sum_{|I|\le m}a_I(x,y)E_I,
\]
where each \(E_I\) is a product of the generators
\(x\partial_x,xY_1,\ldots,xY_{2n-2},x^2T\), then the normal operator at
\(p\in M_0\) is obtained by replacing each coefficient by its boundary value
\(a_I(0,p)\) and interpreting the generators as left-invariant vector fields on the
osculating \(\Theta\)-group at \(p\).  This is the local form of the normal map of
Epstein--Melrose--Mendoza; see
\cite[Section~4]{EpsteinMelroseMendoza1991}.

For the indicial operator one retains only the part of \(N_p(P)\) that acts on pure normal
powers.  Indeed, if
\[
        u=r^\lambda a(y)=x^{2\lambda}a(y),
\]
then
\[
        (x\partial_x)u=2\lambda x^{2\lambda}a(y),
        \qquad
        (xY_j)u=x^{2\lambda+1}Y_ja(y),
        \qquad
        (x^2T)u=x^{2\lambda+2}Ta(y).
\]
Thus all Levi and Reeb \(\Theta\)-derivatives raise the \(x\)-order, and the coefficient
of the first term can only come from the normal generator \(x\partial_x=2r\partial_r\).

Let \(\Delta_0\) denote the complex trace Laplacian of the reference
K\"ahler form \(\omega_0\).  The normal operators of \(\Delta_{\gB}\)
and \(\Delta_0\) agree.  Indeed,
\[
        \omega_B=\omega_0+q,
\]
where \(q\) is the sum of
\(\ddbar\log \phi\) and the logarithmic lower-order
term arising from
\(\log(1+(\psi/\phi)r^{n+1}\log r)\).  In a local
\(\Theta\)-frame the coefficient matrices satisfy
\[
        G_{AB}=G^0_{AB}+H_{AB},
        \qquad H_{AB}\longrightarrow0\quad(x\to0),
\]
because an ordinary smooth covariant two-tensor, as well as the logarithmic
correction and its two ordinary derivatives, has positive \(x\)-order in the
\(\Theta\)-coframe.  Hence
\[
               G^{AB}=(G_0^{-1})^{AB}+o(1).
\]
The coefficients of the metric difference, and all their
\(\Theta\)-derivatives, vanish to positive order. The inverse metric
and volume-density formulas therefore give the same boundary values
for the first-order coefficients of the two Laplacians as well.
Hence, by the definition of the normal operator,
\[
        N_p(\Delta_{\gB})=N_p(\Delta_0).
\]
In Chern--Moser coordinates the boundary model at \(p\) is
\(\{\operatorname{Im}w=|z|^2\}\).  Thus \(N_p(\Delta_0)\) is the scalar
normal operator of the Siegel upper half-space, equivalently of the ball.

It remains only to record the scalar indicial polynomial.  For the unit ball let
\(\delta=1-|z|^2\).  Its Bergman metric is
\[
        g_{\alpha\bar\beta}
        =(n+1)\left(\frac{\delta_{\alpha\beta}}{\delta}
        +\frac{\bar z_\alpha z_\beta}{\delta^2}\right),
        \qquad
        g^{\alpha\bar\beta}
        =\frac{\delta}{n+1}\bigl(\delta_{\alpha\beta}-z_\alpha\bar z_\beta\bigr).
\]
For \(f=\delta^\lambda\),
\[
        f_{\alpha\bar\beta}
        =\lambda(\lambda-1)\delta^{\lambda-2}\bar z_\alpha z_\beta
          -\lambda\delta^{\lambda-1}\delta_{\alpha\beta}.
\]
Contracting and using \(|z|^2=1-\delta\), one obtains
\[
        \Delta_{\gB}(\delta^\lambda)
        =\frac{\lambda(\lambda-n)}{n+1}\delta^\lambda+O(\delta^{\lambda+1}).
\]
Thus, for the complex Laplacian convention fixed in Section~\ref{sec:prelim},
\begin{equation}\label{eq:exact-indicial-polynomial}
        I(\Delta_{\gB})(\lambda)=\frac{\lambda(\lambda-n)}{n+1}.
\end{equation}
Only the two simple roots \(0\) and \(n\) will be used.

\subsection{Indicial vanishing}

By the preceding normal-operator calculation, the Bergman Laplacian is a
\(\Theta\)-differential operator of order two; compare
\cite[Sections~1 and~4]{EpsteinMelroseMendoza1991}.  Its exact scalar indicial
polynomial is \eqref{eq:exact-indicial-polynomial}.  Locally one may write
\[
        \Delta_{\gB}
        =\sum a_{j\alpha k}(x,y)
        (x\partial_x)^j(xY)^\alpha(x^2T)^k,
\]
where the sum has \(\Theta\)-order at most two and all coefficients are
polyhomogeneous, with a well-defined boundary value.  Hence, for every smooth
boundary coefficient \(a(y)\),
\begin{equation}\label{eq:indicial-action-expanded}
        \Delta_{\gB}\bigl(r^\lambda a(y)\bigr)
        =P(\lambda)r^\lambda a(y)
          +\Ophg(x^{2\lambda+1}),
        \qquad
        P(\lambda)=\frac{\lambda(\lambda-n)}{n+1}.
\end{equation}
Indeed, a tangential Levi generator \(xY_j\) raises the \(x\)-order by one,
a Reeb generator \(x^2T\) raises it by two, and replacing a coefficient by
its boundary value introduces an additional factor of \(x\).  Thus the
remainder in \eqref{eq:indicial-action-expanded} has strictly larger
\(x\)-order than \(x^{2\lambda}\).

For logarithms, use
\(r^\lambda(\log r)^q=\partial_\lambda^qr^\lambda\).
Differentiating the indicial identity gives
\begin{align}
\Delta_{\gB}\bigl(r^\lambda(\log r)^q a(y)\bigr)
&=r^\lambda
\left[
        P(\lambda)(\log r)^q
        +qP'(\lambda)(\log r)^{q-1}
\right]a(y)  \notag\\
&\quad
+O\bigl(r^\lambda(\log r)^{q-2}\bigr)
+\Ophg(x^{2\lambda+1}).
\label{eq:log-indicial-action}
\end{align}
The first error term is absent when \(q\le1\).  In particular,
\[
        P^{-1}(0)=\{0,n\},\qquad
        P'(0)=-\frac{n}{n+1}\neq0,\qquad
        P'(n)=\frac{n}{n+1}\neq0.
\]

\begin{lemma}[Indicial vanishing]\label{lem:indicial-bootstrap}
Let \(u\) be polyhomogeneous in \(r\) and \(\log r\) on the strongly
pseudoconvex collar, and suppose
\[
        \Delta_{\gB}u=0.
\]
Then:
\begin{enumerate}[label=(\alph*)]
\item if \(u=O(r^\delta)\) for some \(\delta>0\), then \(u=O(r^n)\);
\item if \(u=O(r^{n+\varepsilon})\) for some \(\varepsilon>0\), then
\[
        u=O(r^M)\qquad\text{for every }M.
\]
\end{enumerate}
\end{lemma}

\begin{proof}
Because \(u\) is polyhomogeneous, if one of the conclusions fails then there is a first
nonzero term in its expansion.  Write the first term as
\begin{equation}\label{eq:first-term}
        u_{\lambda,Q}
        =
        r^\lambda\sum_{q=0}^{Q}a_q(y)(\log r)^q,
        \qquad
        a_Q\not\equiv0.
\end{equation}
All other terms in \(u\) have strictly larger \(x\)-order.  By
\eqref{eq:log-indicial-action}, applying \(\Delta_{\gB}\) to those terms, or taking
the \(\Ophg(x^{2\lambda+1})\) remainder from the displayed first term, cannot
produce a term of \(x\)-order \(2\lambda\).  Thus the coefficient at the first
\(x\)-order can be tested independently.

For (a), the hypothesis \(u=O(r^\delta)\), \(\delta>0\), excludes a leading
\(\lambda=0\) term.  If \(u=O(r^n)\) failed, then either \(0<\lambda<n\), or
\(\lambda=n\) with \(Q\geq1\).  If \(0<\lambda<n\), then \(P(\lambda)\neq0\), and
\eqref{eq:log-indicial-action} gives a nonzero coefficient
\[
        P(\lambda)a_Q(y)
\]
for \(r^\lambda(\log r)^Q\) in \(\Delta_{\gB}u\), contradiction.  If \(\lambda=n\)
and \(Q\geq1\), then \(P(n)=0\) but \(P'(n)\neq0\), and the coefficient of
\(r^n(\log r)^{Q-1}\) in \(\Delta_{\gB}u\) is
\[
        QP'(n)a_Q(y),
\]
again a contradiction.  Thus \(u=O(r^n)\).

For (b), if \(u\) did not vanish to infinite order, the first term would have
\(\lambda>n\).  Then \(P(\lambda)\neq0\), and the coefficient of
\(r^\lambda(\log r)^Q\) in \(\Delta_{\gB}u\) is again \(P(\lambda)a_Q(y)\neq0\),
contradiction.  Hence no first nonzero term exists.
\end{proof}

\begin{remark}
The preceding bootstrap uses the full polyhomogeneous boundary expansion of
\(u\), and hence the \(C^\infty\)-smoothness of the strongly pseudoconvex
boundary patch.  For a boundary of only finite regularity, the argument gives
only finitely many terms in the boundary expansion and does not yield
\[
        u\in\dot C^\infty(X_\Theta).
\]
The Carleman argument in Section~\ref{sec:unique-continuation} requires this
infinite-order boundary flatness.
\end{remark}

\subsection{Higher-dimensional boundary decay}

\begin{proposition}\label{prop:decay-high-dim}
Assume $n\ge3$ and, on the fixed strongly pseudoconvex collar,
\[
        \Delta_{\gB}u=0,
        \qquad
        u=\log\frac{\cB}{\Cn}.
\]
Then
\[
        u=O(r^{n+2}).
\]
\end{proposition}

\begin{proof}
Proposition~\ref{prop:local-boundary-asymptotics} gives
$u=\Ophg(r^2)$, hence $u=O(r)$ by \eqref{eq:phg-to-big-O}.
Since $\Delta_{\gB}u=0$, Lemma~\ref{lem:indicial-bootstrap}(a) yields
$u=O(r^n)=o(r^2)$. Hence, by \eqref{eq:u-J-identity} and
Lemma~\ref{lem:local-error},
\[
        J(v_D)=J(v)+O(r^{n+2})
        =\frac{\pi^n}{n!}+o(r^2).
\]
After shrinking the patch, Lemma~\ref{lem:EX-local-use} improves this to
$J(v_D)=\pi^n/n!+O(r^{n+2})$. Applying
Lemma~\ref{lem:local-error} once more gives
\[
        J(v)=\frac{\pi^n}{n!}+O(r^{n+2}),
        \qquad
        u=\log\bigl(1+O(r^{n+2})\bigr)=O(r^{n+2}).
\]
\end{proof}

\begin{proposition}[Infinite-order vanishing on the strongly pseudoconvex patch]
\label{prop:infinite-vanishing}
Let
\[
        u=\log\frac{\cB}{\Cn}.
\]
Assume that
\[
        \Delta_{\gB}u=0
\]
on the fixed strongly pseudoconvex collar over $M_0$. Then, on this collar,
\[
        u=O(r^M)\qquad\text{for every }M>0.
\]
Moreover, after extending \(u\) by zero to \(x=0\),
\[
        u\in\dot C^\infty(X_\Theta).
\]
\end{proposition}

\begin{proof}
If \(n\ge3\), Proposition~\ref{prop:decay-high-dim} gives
\[
        u=O(r^{n+2}).
\]
If \(n=2\), Proposition~\ref{prop:local-boundary-asymptotics}
gives
\[
        u=\Ophg(r^4),
\]
and hence, by \eqref{eq:phg-to-big-O},
\[
        u=O(r^3).
\]
In both cases there is an \(\varepsilon>0\) such that
\[
        u=O(r^{n+\varepsilon}).
\]
The function $u$ is polyhomogeneous by
Proposition~\ref{prop:local-boundary-asymptotics}.
Lemma~\ref{lem:indicial-bootstrap}(b) therefore gives
$u=O(r^M)$ for every $M$.
Since \(u\) is polyhomogeneous and has no nonzero term at any finite
boundary order, the differentiated remainder condition in the definition of
polyhomogeneity implies that, for every \(\Theta\)-differential monomial
\(Q\) and every \(L>0\),
\[
        Qu=O(x^L).
\]
Here we have rewritten the expansion using \(r=x^2\); this only doubles the
powers and replaces \(\log r\) by \(2\log x\).

It remains to pass from \(\Theta\)-derivatives to ordinary derivatives in
the smooth structure of \(X_\Theta\).  Locally,
\[
        \partial_x=x^{-1}(x\partial_x),
        \qquad
        Y_j=x^{-1}(xY_j),
        \qquad
        T=x^{-2}(x^2T).
\]
Consequently, every ordinary differential monomial \(D\) can be written on
the interior as a finite sum
\[
        D=\sum_\nu c_\nu(x,y)x^{-m_\nu}Q_\nu,
\]
where \(c_\nu\) is smooth up to \(x=0\), \(m_\nu\ge0\), and \(Q_\nu\) is a
\(\Theta\)-differential monomial.  Given any \(N>0\), choose \(L\) larger
than \(N+\max_\nu m_\nu\).  The preceding \(\Theta\)-flatness then gives
\[
        Du=O(x^N).
\]
Thus every ordinary derivative of \(u\) tends to zero to arbitrary order.
Extending \(u\) by zero across \(x=0\) therefore gives a
\(C^\infty\)-function flat at the boundary.  Hence
\[
        u\in\dot C^\infty(X_\Theta).
\]
\end{proof}

\section{Local unique continuation at the ACH boundary}
\label{sec:unique-continuation}

For the applications, set
\[
        u=\log\frac{\cB}{\Cn}.
\]
We prove local unique continuation for functions satisfying
\[
        \Delta_{\gB}u=0,
        \qquad u\in\dot C^\infty(X_\Theta)
\]
on a strongly pseudoconvex collar. For $u=\log(\cB/\Cn)$, these
conditions follow from Lemma~\ref{lem:scalar-pde} and
Proposition~\ref{prop:infinite-vanishing}.

Following the convexity mechanism of Vasy--Wunsch
\cite{VasyWunsch2005}, we construct a tangentially localized weight
$\tau$ on a region $E_S$ and verify condition~(28.2.6) of
\cite[Theorem~28.2.3, p.~237]{HormanderIV}. Uniform normalized ACH
charts give common constants in the local Carleman estimates, which
we sum using a partition of unity. Cutoffs then remove the tangential
and upper ends. Throughout, $g_B^{\mathrm R}$ is associated with
$\omega_B$, and $x=r^{1/2}$ is the boundary defining function on
$X_\Theta$.

\begin{lemma}
\label{lem:local-ach-product}
Let \(\Gamma_1\Subset M_0\).  After shrinking, choose a boundary patch
\(\Gamma_0\) satisfying
\[
        \overline{\Gamma_0}\Subset\Gamma_1.
\]
Let
\[
        \omega_0:=(n+1)\ddbar(-\log r),
\]
and \(g_0^{\mathrm R}\) be its associated real Riemannian metric.  There
exist a constant \(c_B>0\), a smooth boundary defining function
\(\widetilde x\) on the square-root compactification over \(\Gamma_0\), and constants
\(c,C,c_0>0\) such that
\begin{equation}\label{eq:reference-defining-function-comparison}
        0<c\le \frac{\widetilde x}{x}\le C.
\end{equation}
The quotient \(\widetilde x/x\), its reciprocal, and all of their derivatives
are bounded in normalized \(\Theta\)-coordinates through every fixed finite
order.

Put
\[
        \widehat g_0:=c_Bg_0^{\mathrm R},
        \qquad
        \widehat g:=c_Bg_B^{\mathrm R},
        \qquad
        s:=-\log\widetilde x.
\]
For some \(S_0>0\), the coordinates obtained from the
\(\widehat g_0\)-gradient flow of \(s\) identify the collar with
\[
        (S_0,\infty)_s\times \Gamma_0
\]
and put the \emph{reference} metric in the exact geodesic normal form
\begin{equation}\label{eq:local-ach-product-form}
        \widehat g_0=ds^2+a_s^0.
\end{equation}
Let
\[
        b_s:=\widehat g\big|_{T(\{s\}\times \Gamma_0)}
\]
be the restriction of the rescaled Bergman metric \(\widehat g\) to the
level set of \(s\).
Choose a contact form \(\theta\) on \(\Gamma_0\), put \(H=\ker\theta\), and fix a
positive-definite metric \(h_0\) on \(H\).  After increasing \(S_0\), one has
\begin{equation}\label{eq:ach-tangential-comparison}
        C^{-1}\bigl(e^{2s}h_0+e^{4s}\theta^2\bigr)
        \le b_s
        \le
        C\bigl(e^{2s}h_0+e^{4s}\theta^2\bigr).
\end{equation}
Moreover,
\begin{equation}\label{eq:reference-s-gradient-hessian}
\begin{split}
        \frac12\le |ds|_{\widehat g}\le2, \qquad
        \nabla_{\widehat g}^2s(X,X)
        &\ge c_0|X|_{\widehat g}^2
        \quad\text{if }ds(X)=0,\\
        \left|
        \nabla_{\widehat g}^2s
        (\nabla_{\widehat g}s,W)
        \right|
        &\le Ce^{-s}|W|_{\widehat g}.
\end{split}
\end{equation}
All coefficients of \(\widehat g\) and \(\widehat g^{-1}\), together with
their derivatives through every fixed finite order, are uniformly bounded in
the normalized ACH frame
\[
        \partial_s,\qquad
        e^{-s}Y_1,\ldots,e^{-s}Y_{2n-2},\qquad
        e^{-2s}T,
\]
where the \(Y_j\)'s span \(H\) and \(T\) is the Reeb field.  Finally, for a
fixed smooth positive density \(d\mu_0\) on \(\Gamma_0\),
\begin{equation}\label{eq:ach-volume-comparison}
        C^{-1}e^{2ns}\,ds\,d\mu_0
        \le
        dV_{\widehat g}
        \le
        Ce^{2ns}\,ds\,d\mu_0.
\end{equation}
\end{lemma}

\begin{proof}
By \eqref{eq:local-ACH-leading}, \(g_0^{\mathrm R}\) is a smooth
positive-definite \(\Theta\)-metric with complex-hyperbolic boundary value.
In particular, there is a universal constant \(\kappa_n>0\), independent of
the point of \(\Gamma_0\), such that
\[
        \left.\left|\frac{dx}{x}\right|_{g_0^{\mathrm R}}^2
        \right|_{x=0}=\kappa_n.
\]
The convention above gives
\(\kappa_n=1/(2(n+1))\), since the leading coefficient of
\((dx/x)^2\) in \(g_0^{\mathrm R}\) is \(2(n+1)\).
Choose \(c_B=\kappa_n=1/(2(n+1))\). Then
\[
        \left.\left|\frac{dx}{x}\right|_{\widehat g_0}^2
        \right|_{x=0}=1.
\]
Equivalently, the smooth \(\Theta\)-metric
\[
        \overline g_0:=4\widehat g_0
\]
satisfies hypothesis~(H1) of Guillarmou--S\'a Barreto:
\[
        \left.\left|\frac{dx}{x}\right|_{\overline g_0}
        \right|_{x=0}=\frac12.
\]
Moreover, because \(\overline g_0\) is a smooth ACH \(\Theta\)-metric,
\(x^4\overline g_0|_{T\partial X_\Theta}\) is the square of a contact
form \(\theta_x\) in the CR contact conformal class.  Thus the boundary
representative required in Lemma~3.1 of Guillarmou--S\'a Barreto is already
fixed by \(x\); below, the choice \(\omega_{\partial}=0\) preserves this
representative.  Only hypothesis~(H1) and the proof of their Lemma~3.1 are
used here; the subsequent global ACH hypotheses in that paper play no role.

We now apply that construction on a boundary patch.
Choose \(\Gamma_0'\) with
\(\overline \Gamma_0\Subset \Gamma_0'\Subset\Gamma_1\), and work on a collar of \(\Gamma_0'\).
For a boundary defining function \(x\), let \(X_x\) be the
\(\overline g_0\)-dual of \(dx/x\).  Seeking
\(\widetilde x=e^\omega x\) with prescribed smooth boundary value
\(\omega|_{x=0}=\omega_{\partial}\), the equation
\[
        \left|\frac{d\widetilde x}{\widetilde x}\right|_{\overline g_0}
        =\frac12
\]
is exactly
\begin{equation}\label{eq:GSB-local-eikonal}
        2\frac{X_x(\omega)}{x}
        +\frac{|d\omega|_{\overline g_0}^2}{x}
        =\frac{\frac14-|X_x|_{\overline g_0}^2}{x},
\end{equation}
which is \cite[equation~(3.3)]{GuillarmouSaBarreto2008}.  All three quotients
in \eqref{eq:GSB-local-eikonal} are smooth up to \(x=0\): the middle numerator
is \(O(x^2)\), and the right numerator vanishes at \(x=0\) by~(H1).  Moreover,
\(X_x/x\) is a smooth vector field transverse to the boundary.  Thus
\eqref{eq:GSB-local-eikonal} is a smooth noncharacteristic first-order Cauchy
problem.  The method of characteristics gives a unique smooth solution on a
collar of \(\Gamma_0'\); since \(\overline \Gamma_0\) is compact, one collar width works
uniformly over \(\Gamma_0\).

This is the local form of \cite[Lemma~3.1]{GuillarmouSaBarreto2008}.
Taking \(\omega_{\partial}=0\), the quotient \(\widetilde x/x=e^\omega\), its
reciprocal, and all derivatives of each fixed finite order are bounded on the
smaller collar, proving \eqref{eq:reference-defining-function-comparison}.

Let \(X_{\widetilde x}\) denote the
\(\overline g_0\)-dual of \(d\widetilde x/\widetilde x\).  After decreasing
the collar width, the integral curves of
\(4X_{\widetilde x}/\widetilde x\) issuing from \(\overline \Gamma_0\) remain in
the collar over \(\Gamma_0'\).  The flow construction following
\cite[Lemma~3.1]{GuillarmouSaBarreto2008} therefore gives
\[
        \overline g_0
        =4\frac{d\widetilde x^2}{\widetilde x^2}
          +\frac{h(\widetilde x)}{\widetilde x^2},
\]
with no mixed normal--tangential term; this is their equation~(3.4).  Dividing
by \(4\) and setting \(s=-\log\widetilde x\) gives
\[
        \widehat g_0=ds^2+a_s^0,
\]
which is \eqref{eq:local-ach-product-form}.

In the coframe
\[
        ds,\qquad e^s\eta^1,\ldots,e^s\eta^{2n-2},
        \qquad e^{2s}\theta,
\]
the coefficient matrix of \(\widehat g_0\) converges smoothly to the
positive-definite complex-hyperbolic model matrix.  Consequently
\begin{equation}\label{eq:reference-tangential-comparison}
        C^{-1}\bigl(e^{2s}h_0+e^{4s}\theta^2\bigr)
        \le a_s^0
        \le
        C\bigl(e^{2s}h_0+e^{4s}\theta^2\bigr),
\end{equation}
and, after increasing \(S_0\),
\begin{equation}\label{eq:reference-second-fundamental-form}
        \frac12(\partial_sa_s^0)(X,X)
        \ge 2c_0|X|_{\widehat g_0}^2
        \qquad
        (X\in T(\{s\}\times \Gamma_0)).
\end{equation}
For the metric \(\widehat g_0=ds^2+a_s^0\) in this geodesic normal form,
\[
        \nabla_{\widehat g_0}^2s(X,X)
        =
        \frac12(\partial_sa_s^0)(X,X),
        \qquad
        \nabla_{\widehat g_0}^2s(\partial_s,\cdot)=0.
\]

We now compare \(\widehat g\) with \(\widehat g_0\).
On the fixed collar,
\[
\begin{aligned}
        \log K
        &=
        -(n+1)\log r+\log\phi\\
        &\quad+
        \log\left(
        1+\frac{\psi}{\phi}r^{n+1}\log r
        \right).
\end{aligned}
\]
Hence
\begin{equation}\label{eq:actual-reference-metric-difference}
        \widehat g-\widehat g_0
        =
        c_B\bigl(\ddbar\log\phi\bigr)^{\mathrm R}
        +
        c_B\left(
        \ddbar
        \log\left(
        1+\frac{\psi}{\phi}r^{n+1}\log r
        \right)
        \right)^{\mathrm R}.
\end{equation}
An ordinary smooth covariant two-tensor has coefficients vanishing to
positive order in the \(\Theta\)-coframe.  The second term in
\eqref{eq:actual-reference-metric-difference} has still higher positive
order.  Therefore, in the normalized \(\Theta\)-frame associated with this
geodesic normal form,
\begin{equation}\label{eq:actual-reference-decay}
        \left|
        \nabla_{\widehat g_0}^{\,j}
        (\widehat g-\widehat g_0)
        \right|_{\widehat g_0}
        \le C_je^{-s},
        \qquad j=0,1,\ldots,L,
\end{equation}
for every fixed finite \(L\), after increasing \(S_0\).  Logarithmic factors
cause no difficulty, since \(s^qe^{-ms}\le C_{q,m}e^{-s}\) when \(m\ge2\)
and \(s\) is large.

The estimates for inverses, volume forms, and Levi-Civita connections follow
from \eqref{eq:actual-reference-decay}.  In particular,
\[
        |\widehat g^{-1}-\widehat g_0^{-1}|_{\widehat g_0}
        +
        |\nabla^{\widehat g}-\nabla^{\widehat g_0}|_{\widehat g_0}
        \le Ce^{-s}.
\]
The tangential comparison
\eqref{eq:ach-tangential-comparison} follows from
\eqref{eq:reference-tangential-comparison}.  Also
\[
        |ds|_{\widehat g}=1+O(e^{-s}).
\]
For \(ds(X)=0\),
\[
\begin{aligned}
        \nabla_{\widehat g}^2s(X,X)
        &=
        \nabla_{\widehat g_0}^2s(X,X)
        -
        ds\bigl(
        (\nabla^{\widehat g}-\nabla^{\widehat g_0})_XX
        \bigr)\\
        &\ge
        2c_0|X|_{\widehat g_0}^2
        -
        Ce^{-s}|X|_{\widehat g_0}^2,
\end{aligned}
\]
which gives the first Hessian inequality in
\eqref{eq:reference-s-gradient-hessian}.  Since
\(\nabla_{\widehat g}s=\partial_s+O(e^{-s})\) in the reference bounded frame
and
\(\nabla_{\widehat g_0}^2s(\partial_s,\cdot)=0\), the same comparison gives
the second inequality in
\eqref{eq:reference-s-gradient-hessian}.

The normalized-coordinate bounds follow from the smooth ACH bounded geometry
of \(\widehat g_0\) and \eqref{eq:actual-reference-decay}.  Finally, there
are \(2n-2\) horizontal directions with length scale \(e^s\), one Reeb
direction with length scale \(e^{2s}\), and one normal direction with bounded
length.  This gives \eqref{eq:ach-volume-comparison}.
\end{proof}

For later use, we make the phrase \emph{normalized ACH coordinate chart}
precise.  Fix a finite contact-adapted atlas on a neighborhood of
\(\overline \Gamma_0\).  On each member of this atlas choose smooth vector fields
\[
        Y_1,\ldots,Y_{2n-2},T,
\]
where the \(Y_j\)'s span the Levi distribution and \(T\) is the Reeb field.
If \(q=(s_q,y_q)\) lies in the corresponding collar, let
\(\mathcal F_{y_q}(z,t)\) denote the local map obtained by composing, in a
fixed order, the time-\(z_j\) flows of the \(Y_j\)'s and the time-\(t\) flow
of \(T\), starting at \(y_q\).  We call a chart of the form
\begin{equation}\label{eq:normalized-ach-flow-chart}
        \Phi_q(\sigma,z,t)
        =
        \bigl(
        s_q+\sigma,
        \mathcal F_{y_q}(e^{-s_q}z,e^{-2s_q}t)
        \bigr)
\end{equation}
a normalized ACH flow chart centered at \(q\).

At the center, the differential of \(\Phi_q\) identifies the Euclidean
coordinate frame with
\[
        \partial_s,
        \qquad e^{-s_q}Y_1,\ldots,e^{-s_q}Y_{2n-2},
        \qquad e^{-2s_q}T.
\]
Put $\lambda=e^{-s_q}$ and write $E_j=\lambda Y_j$ and
$E_T=\lambda^2T$. Work in finitely many nested background charts whose
smaller members cover $\overline\Gamma_0$. The unscaled flow maps
$\mathcal F_y$ are injective on a common neighborhood of the origin;
for large $s_q$, the scaled parameter boxes lie in that neighborhood.
Moreover, the structure coefficients of the frame $(E_j,E_T)$ and
their derivatives along this frame are uniformly bounded. Indeed,
the horizontal and Reeb components of $[E_i,E_j]$ have coefficients
$O(\lambda)$ and $O(1)$ in this frame, respectively, while those of
$[E_i,E_T]$ have coefficients $O(\lambda^2)$ and $O(\lambda)$.

Express the differential of each composed flow in the scaled frame
at its endpoint. The pushforward equations are linear ordinary
differential equations with uniformly bounded coefficients and
initial matrix equal to the identity. On a sufficiently small fixed
parameter ball their solutions and inverses are therefore uniformly
bounded. Differentiating these equations gives the same conclusion
for every fixed number of parameter derivatives. Since
$|s-s_q|=|\sigma|$ is bounded on this ball, the scales at $s$ and
$s_q$ are uniformly comparable. Combining these frame comparisons
with the coefficient estimates above proves that $\Phi_q^*\widehat g$
and its inverse are uniformly elliptic and have uniformly bounded
derivatives of every fixed finite order, on a radius independent of $q$.

\begin{lemma}
\label{lem:convex-graph-end}
Choose a smoothly bounded connected open set
\[
        Y\Subset \Gamma_0
\]
and a smooth defining function \(\delta\) for \(Y\), normalized so that
\[
        Y=\{\delta>0\},
        \qquad
        d\delta\neq0\text{ on }\partial Y,
        \qquad
        0<\delta<1\text{ on }Y.
\]
Put
\[
        \varphi=-\log\delta.
\]
For \(S\ge S_0\), define
\[
        E_S
        :=
        \left\{
        (s,y)\in(S_0,\infty)\times Y:
        s>S+\varphi(y)
        \right\}
\]
and
\begin{equation}\label{eq:graph-weight}
        \tau(s,y)
        :=
        s-S-\varphi(y).
\end{equation}
For all sufficiently large \(S\), the following statements hold.

\begin{enumerate}[label=(\roman*)]
\item On \(E_S\),
\begin{equation}\label{eq:x-delta-comparison}
        \widetilde x=e^{-s}\le e^{-S}\delta(y).
\end{equation}

\item There are constants \(c_1>0\) and \(C_1>0\), independent of the
point of \(E_S\), such that
\begin{equation}\label{eq:graph-weight-gradient}
        \frac12\le |d\tau|_{\widehat g}\le2
\end{equation}
and
\begin{equation}\label{eq:graph-tangential-convexity}
        \nabla_{\widehat g}^2\tau(V,V)
        \ge c_1|V|_{\widehat g}^2
        \qquad
        \text{whenever }d\tau(V)=0.
\end{equation}
Moreover,
\begin{equation}\label{eq:graph-normal-hessian}
        \left|
        \nabla_{\widehat g}^2\tau
        (\nabla_{\widehat g}\tau,\nabla_{\widehat g}\tau)
        \right|
        \le C_1e^{-S}.
\end{equation}

\item The hypersurface
\[
        \mathcal N_S:=\{\tau=0\}
\]
is complete with its induced metric.

\item There are fixed constants \(R_0>0\) and
\(0<r_{\mathrm{loc}}<R_0\) such that, for every \(q\in E_S\), there is a
normalized ACH flow chart in the sense of
\eqref{eq:normalized-ach-flow-chart},
\[
        \Phi_q:B(0,2R_0)\longrightarrow (S_0,\infty)\times Y,
        \qquad \Phi_q(0)=q,
\]
whose image lies in the ambient collar and such that
\begin{equation}\label{eq:chart-ball-inclusion}
        B_{\widehat g}(q,2r_{\mathrm{loc}})
        \Subset
        \Phi_q\bigl(B(0,R_0)\bigr)
        \Subset
        \Phi_q\bigl(B(0,2R_0)\bigr).
\end{equation}
The chart need not be contained in \(E_S\).  In these charts the pulled-back
metric and its inverse are uniformly elliptic, their derivatives have uniform bounds of every fixed finite order, and every
positive-order derivative of \(\tau\circ\Phi_q\) has a uniform bound of each
fixed finite order.  Moreover,
the estimates
\eqref{eq:graph-weight-gradient}--\eqref{eq:graph-normal-hessian}, with
possibly different uniform constants, hold throughout \(\Phi_q(B(0,2R_0))\).
\end{enumerate}
\end{lemma}

\begin{proof}
The first assertion is immediate from
\[
        s>S-\log\delta(y).
\]

Let \(Z_1,\ldots,Z_{2n-1}\) be a fixed smooth frame on a neighborhood of
\(\overline Y\).  For every multiindex \(\alpha\neq0\),
\[
        |Z^\alpha\varphi|
        \le C_\alpha\delta^{-|\alpha|}.
\]
On \(E_S\), \eqref{eq:x-delta-comparison} gives
\[
        e^{-s}\delta^{-1}\le e^{-S}.
\]
Since the bounded-length tangential fields are of the form
\[
        e^{-s}Y_j,\qquad e^{-2s}T,
\]
the preceding estimates and the uniform ACH connection bounds imply
\begin{equation}\label{eq:small-phi-derivatives}
        |d\varphi|_{\widehat g}
        +
        |\nabla_{\widehat g}^2\varphi|_{\widehat g}
        \le Ce^{-S}.
\end{equation}
The same argument gives uniform estimates for every fixed finite number of
higher covariant derivatives.

By Lemma~\ref{lem:local-ach-product},
\begin{equation}\label{eq:hessian-s-convexity}
        \nabla_{\widehat g}^2s(X,X)
        \ge c_0|X|_{\widehat g}^2
        \qquad\text{whenever }ds(X)=0.
\end{equation}

Since
\[
        d\tau=ds-d\varphi,
\]
estimate \eqref{eq:graph-weight-gradient} follows from
\eqref{eq:small-phi-derivatives} and
\eqref{eq:reference-s-gradient-hessian}.  Let \(V\) satisfy
\(d\tau(V)=0\).  Decompose it orthogonally as
\[
        V=a\nabla_{\widehat g}s+X,
        \qquad
        ds(X)=0.
\]
Then
\[
        0=d\tau(V)
        =
        a|ds|_{\widehat g}^2
        -
        a\,d\varphi(\nabla_{\widehat g}s)
        -
        d\varphi(X).
\]
Using \eqref{eq:small-phi-derivatives} and
\(\frac12\le |ds|_{\widehat g}\le2\), we obtain
\[
        |a|\le Ce^{-S}|X|_{\widehat g}.
\]
The mixed Hessian estimate in
\eqref{eq:reference-s-gradient-hessian} gives
\[
\begin{aligned}
        \nabla_{\widehat g}^2s(V,V)
        &=
        \nabla_{\widehat g}^2s(X,X)
        +
        2a\nabla_{\widehat g}^2s
        (\nabla_{\widehat g}s,X)\\
        &\quad+
        a^2\nabla_{\widehat g}^2s
        (\nabla_{\widehat g}s,\nabla_{\widehat g}s)\\
        &\ge
        c_0|X|_{\widehat g}^2
        -
        Ce^{-S}|V|_{\widehat g}^2.
\end{aligned}
\]
Subtracting
\(\nabla_{\widehat g}^2\varphi(V,V)\) and using
\eqref{eq:small-phi-derivatives} yields
\[
        \nabla_{\widehat g}^2\tau(V,V)
        \ge
        c_0|X|_{\widehat g}^2
        -
        Ce^{-S}|V|_{\widehat g}^2.
\]
For \(S\) sufficiently large, \(|V|_{\widehat g}\) and
\(|X|_{\widehat g}\) are uniformly equivalent, and
\eqref{eq:graph-tangential-convexity} follows.

Furthermore,
\[
        \nabla_{\widehat g}\tau
        =
        \nabla_{\widehat g}s-\nabla_{\widehat g}\varphi.
\]
The mixed Hessian estimate for \(s\), together with
\eqref{eq:small-phi-derivatives}, therefore gives
\[
        \left|
        \nabla_{\widehat g}^2\tau
        (\nabla_{\widehat g}\tau,\nabla_{\widehat g}\tau)
        \right|
        \le Ce^{-S},
\]
which is \eqref{eq:graph-normal-hessian}.

The parametrization $y\mapsto(S+\varphi(y),y)$ identifies
$\mathcal N_S$ with $Y$. For $W\in T_yY$, its tangent lift is
$\widetilde W=W+d\varphi(W)\partial_s$.
By \eqref{eq:small-phi-derivatives} and the metric comparison, the
induced metric satisfies
\[
        g_{\mathcal N_S}(W,W)
        \ge c\,b_{S+\varphi}(W,W)
        \ge ce^{2S}\delta^{-2}|W|_{h_*}^2,
\]
where $h_*$ is a fixed smooth metric on $\overline Y$.
Since $|d\delta|_{h_*}$ is bounded, the length of a curve tending
to $\partial Y$ is bounded below by a positive multiple of the
total variation of $\log\delta$, which is infinite. Thus
$\mathcal N_S$ is complete.

Finally, let \(q=(s_q,y_q)\in E_S\).  Since
\[
        e^{-s_q}\delta(y_q)^{-1}\le e^{-S},
\]
and a smooth defining function is comparable to the distance to
\(\partial Y\) near the boundary, the horizontal displacement
\(O(e^{-s_q})\) and the Reeb displacement \(O(e^{-2s_q})\) in the
normalized flow chart are both smaller than a fixed small multiple of
\(\operatorname{dist}(y_q,\partial Y)\) once \(S\) is large.  Using the
finite contact-adapted atlas fixed above, we may therefore choose a single
\(R_0>0\) such that every chart
\[
        \Phi_q:B(0,2R_0)\longrightarrow(S_0,\infty)\times Y
\]
is well defined and remains in the ambient collar.

On this fixed coordinate ball,
\[
        |s-s_q|\le C,
        \qquad
        C^{-1}\delta(y_q)\le\delta(y)\le C\delta(y_q),
\]
and hence
\[
        e^{-s}\delta(y)^{-1}\le Ce^{-S}.
\]
The estimates used in \eqref{eq:small-phi-derivatives}, and their
higher-order analogues through every fixed finite order, therefore hold
uniformly on the full ambient chart, including its part outside \(E_S\).
The polyhomogeneous ACH expansion and the construction of the normalized
flow charts give uniform ellipticity and uniform bounds for all fixed finite
coordinate derivatives of the pulled-back metric and its inverse.  The same
computations proving
\eqref{eq:graph-weight-gradient}--\eqref{eq:graph-normal-hessian} hence hold
throughout the chart.

Uniform ellipticity also gives a fixed positive lower bound for the
\(\widehat g\)-distance from \(q\) to the image of
\(\partial B(0,R_0)\).  Decreasing a fixed
\(r_{\mathrm{loc}}>0\), independently of \(q\), we therefore obtain
\[
        B_{\widehat g}(q,2r_{\mathrm{loc}})
        \Subset
        \Phi_q\bigl(B(0,R_0)\bigr),
\]
which is \eqref{eq:chart-ball-inclusion} and completes the proof.
\end{proof}

\begin{proposition}
\label{prop:theta-carleman}
Let \(E_S\) and \(\tau\) be as in
Lemma~\ref{lem:convex-graph-end}, with \(S\) sufficiently large.  Set
\[
        P=-\Delta_{\widehat g},
\]
where the sign is chosen so that
\[
        \sigma_2(P)(z,\xi)=|\xi|_{\widehat g}^2.
\]
Then there are constants \(C>0\) and \(\gamma_0>0\) such that
\begin{equation}\label{eq:graph-carleman}
        \gamma^3
        \|e^{\gamma\tau}w\|_{L^2(E_S,dV_{\widehat g})}^2
        +
        \gamma
        \|e^{\gamma\tau}\nabla_{\widehat g}w\|_
        {L^2(E_S,dV_{\widehat g})}^2
        \le
        C
        \|e^{\gamma\tau}Pw\|_{L^2(E_S,dV_{\widehat g})}^2
\end{equation}
for every \(w\in C_c^\infty(E_S)\) and every
\(\gamma\ge\gamma_0\).
\end{proposition}

\begin{proof}
We first verify strong pseudoconvexity of the weight \(\tau\).  The principal
symbol of the conjugated operator
\[
        e^{\gamma\tau}Pe^{-\gamma\tau}
\]
is
\[
        p_{\tau,\gamma}(z,\xi)
        =
        |\xi|_{\widehat g}^2
        -
        \gamma^2|d\tau|_{\widehat g}^2
        +
        2\sqrt{-1}\gamma
        \langle\xi,d\tau\rangle_{\widehat g}.
\]
On the complex characteristic set
\[
        \operatorname{Re}p_{\tau,\gamma}
        =
        \operatorname{Im}p_{\tau,\gamma}
        =
        0,
\]
we have
\begin{equation}\label{eq:conjugated-characteristic-set}
        \langle\xi,d\tau\rangle_{\widehat g}=0,
        \qquad
        |\xi|_{\widehat g}^2
        =
        \gamma^2|d\tau|_{\widehat g}^2.
\end{equation}
The invariant Poisson-bracket computation gives
\begin{equation}\label{eq:graph-poisson-bracket}
\begin{split}
        \{
        \operatorname{Re}p_{\tau,\gamma},
        \operatorname{Im}p_{\tau,\gamma}
        \}
        =
        4\gamma\bigl(
        &\nabla_{\widehat g}^2\tau
          (\xi^\sharp,\xi^\sharp)  \\
        &+
        \gamma^2
        \nabla_{\widehat g}^2\tau
        (\nabla_{\widehat g}\tau,\nabla_{\widehat g}\tau)
        \bigr).
\end{split}
\end{equation}
By the first equation in
\eqref{eq:conjugated-characteristic-set}, \(\xi^\sharp\) is tangent to a
level hypersurface of \(\tau\).  Lemma~\ref{lem:convex-graph-end} therefore
gives
\[
        \nabla_{\widehat g}^2\tau
        (\xi^\sharp,\xi^\sharp)
        \ge
        c_1|\xi|_{\widehat g}^2.
\]
Using \eqref{eq:graph-weight-gradient},
\eqref{eq:graph-normal-hessian}, and the second equation in
\eqref{eq:conjugated-characteristic-set}, we obtain
\[
\begin{aligned}
        &
        \nabla_{\widehat g}^2\tau
        (\xi^\sharp,\xi^\sharp)
        +
        \gamma^2
        \nabla_{\widehat g}^2\tau
        (\nabla_{\widehat g}\tau,\nabla_{\widehat g}\tau) \\
        &\hspace{2cm}
        \ge
        \left(c_2-Ce^{-S}\right)\gamma^2.
\end{aligned}
\]
After increasing \(S\), the right-hand side is bounded below by
\(c_3\gamma^2\).  Consequently,
\begin{equation}\label{eq:uniform-strong-pseudoconvexity}
        \{
        \operatorname{Re}p_{\tau,\gamma},
        \operatorname{Im}p_{\tau,\gamma}
        \}
        \ge
        c_4\gamma^3
\end{equation}
on the conjugated characteristic set, uniformly on \(E_S\).

Let \(v\in C_c^\infty(E_S)\) be supported in
\[
        B_{\widehat g}(q,r_{\mathrm{loc}}).
\]
Use the chart in Lemma~\ref{lem:convex-graph-end}(iv), and define
\[
        \widetilde v=v\circ\Phi_q,
        \qquad
        P_q=\Phi_q^*P,
        \qquad
        \vartheta_q=\tau\circ\Phi_q-\tau(q).
\]
By \eqref{eq:chart-ball-inclusion},
\[
        \operatorname{supp}\widetilde v
        \Subset B(0,R_0).
\]
We apply the local estimate on the single fixed compact set
\[
        \mathcal K:=\overline{B(0,3R_0/2)}
        \Subset X_1:=B(0,7R_0/4)
        \Subset B(0,2R_0),
\]
which is independent of \(q\) and \(v\).  The subtraction of the constant
\(\tau(q)\) does not change any derivative of the weight or its
pseudoconvexity.  All symbol estimates below are proved on the fixed open set
\(X_1\); in particular they hold on a neighborhood of \(\mathcal K\).

In this fixed-chart calculation, \(x\in X_1\) denotes the full
coordinate vector and \(D_j=-\sqrt{-1}\,\partial_{x_j}\).
Let \(p_q\) be the real quadratic principal symbol of \(P_q\), and set
\[
        \mathfrak p_{q,\gamma}(x,\xi)
        :=p_q\bigl(x,\xi+\sqrt{-1}\gamma
        d\vartheta_q(x)\bigr).
\]
Thus \(\mathfrak p_{q,\gamma}=p_{q,\vartheta_q,\gamma}\).  Define
\begin{equation}\label{eq:hormander-bracket-term}
\begin{split}
        \mathfrak b_q(x,\xi,\gamma)
        &:=\frac{1}{2\sqrt{-1}}
        \left\{
        p_q\bigl(x,\xi-\sqrt{-1}\gamma d\vartheta_q(x)\bigr),
        p_q\bigl(x,\xi+\sqrt{-1}\gamma d\vartheta_q(x)\bigr)
        \right\}                                                   \\
        &=\left\{
        \operatorname{Re}\mathfrak p_{q,\gamma},
        \operatorname{Im}\mathfrak p_{q,\gamma}
        \right\}.
\end{split}
\end{equation}
The second equality uses that \(p_q\) is real.  Pulling back
\eqref{eq:uniform-strong-pseudoconvexity} and using the full-chart estimates
in Lemma~\ref{lem:convex-graph-end}(iv), after increasing \(S\) once more if
necessary, gives
\begin{equation}\label{eq:uniform-chart-pseudoconvexity}
        \mathfrak b_q(x,\xi,\gamma)
        \ge c_5\gamma^3
\end{equation}
on \(\{\mathfrak p_{q,\gamma}=0\}\) over \(X_1\), uniformly in \(q\) and
\(\gamma\ge1\).  The operators \(P_q\) are uniformly elliptic on \(X_1\), and
the relevant coefficient and weight seminorms are uniformly bounded there.

We next verify H\"ormander's full quantitative condition~(28.2.6), rather
than only its restriction to the conjugated characteristic set.  All norms in
the following argument are Euclidean coordinate norms; uniform ellipticity
makes them uniformly equivalent to the corresponding pulled-back metric
norms.  Put
\[
        \Lambda=(|\xi|^2+\gamma^2)^{1/2}.
\]
The bounds for \(d\vartheta_q\), the uniform \(C^2\)-bounds for the metric and
weight, and the bracket formula \eqref{eq:graph-poisson-bracket} imply
\begin{equation}\label{eq:hormander-rough-bounds}
\begin{split}
        &\sum_{|\alpha|<2}
        \gamma^{2(2-|\alpha|)-1}
        \left|
        \bigl(\xi+\sqrt{-1}\gamma d\vartheta_q(x)\bigr)^\alpha
        \right|^2
        \le C\Lambda^3,                                             \\
        &|\mathfrak b_q(x,\xi,\gamma)|\le C\Lambda^3.
\end{split}
\end{equation}
The constants are independent of \(q\).

We claim that there are \(\varepsilon>0\), \(c>0\), and a sufficiently
large \(S\), all independent of \(q\), such that
\begin{equation}\label{eq:near-characteristic-bracket}
        |\mathfrak p_{q,\gamma}(x,\xi)|\le\varepsilon\Lambda^2
        \quad\Longrightarrow\quad
        \mathfrak b_q(x,\xi,\gamma)\ge c\Lambda^3.
\end{equation}
To prove the claim, write \(h_q=\Phi_q^*\widehat g\),
\(N_q=\nabla_{h_q}\vartheta_q\), and
\[
\begin{split}
        \operatorname{Re}\mathfrak p_{q,\gamma}
        &=|\xi|_{h_q}^2-
          \gamma^2|d\vartheta_q|_{h_q}^2,                            \\
        \operatorname{Im}\mathfrak p_{q,\gamma}
        &=2\gamma\langle\xi,d\vartheta_q\rangle_{h_q}.
\end{split}
\]
If \(\gamma/\Lambda\) were smaller than a fixed sufficiently small number,
uniform ellipticity and the upper bound for \(|d\vartheta_q|_{h_q}\) would
give
\(
 |\operatorname{Re}\mathfrak p_{q,\gamma}|\ge c_0\Lambda^2
\).
Thus, after choosing \(\varepsilon\) sufficiently small,
\begin{equation}\label{eq:gamma-lambda-comparable}
        \gamma\ge c_1\Lambda,
        \qquad
        |\xi|_{h_q}\asymp\gamma
\end{equation}
whenever the hypothesis in \eqref{eq:near-characteristic-bracket} holds.
Decompose
\[
        \xi^\sharp=V+aN_q,
        \qquad
        d\vartheta_q(V)=0.
\]
The estimate for the imaginary part and
\eqref{eq:gamma-lambda-comparable} give
\[
        |a|\le C\varepsilon\gamma,
        \qquad
        |V|_{h_q}\ge c_2\gamma
\]
when \(\varepsilon\) is small.  By
\eqref{eq:graph-tangential-convexity}, the uniform Hessian bounds, and
\eqref{eq:graph-normal-hessian},
\begin{align*}
        &\nabla_{h_q}^2\vartheta_q(\xi^\sharp,\xi^\sharp)
        +\gamma^2\nabla_{h_q}^2\vartheta_q(N_q,N_q)                  \\
        &\quad\ge
        c_3|V|_{h_q}^2
        -C|a|\,|V|_{h_q}-Ca^2-Ce^{-S}\gamma^2                       \\
        &\quad\ge c_4\gamma^2,
\end{align*}
provided first \(\varepsilon\) is chosen small and then \(S\) is chosen
large.  Formula \eqref{eq:graph-poisson-bracket} now yields
\(
 \mathfrak b_q\ge c_5\gamma^3
\), which, together with
\eqref{eq:gamma-lambda-comparable}, proves
\eqref{eq:near-characteristic-bracket}.

We split phase space into the two regions determined by
\eqref{eq:near-characteristic-bracket}.  In the near-characteristic region,
\eqref{eq:hormander-rough-bounds} and
\eqref{eq:near-characteristic-bracket} show that the left-hand side of
H\"ormander's condition~(28.2.6) is controlled by \(\mathfrak b_q\).  In the
complementary region,
\[
        |\mathfrak p_{q,\gamma}|^2+\mathfrak b_q
        \ge\varepsilon^2\Lambda^4-C\Lambda^3
        \ge \frac{\varepsilon^2}{2}\Lambda^4
\]
for \(\gamma\ge\gamma_1\), after increasing the common threshold
\(\gamma_1\).  Together with \eqref{eq:hormander-rough-bounds}, this proves
that there is a constant \(C_H>0\), independent of \(q\), such that
\begin{equation}\label{eq:hormander-286-uniform}
\begin{split}
        &\sum_{|\alpha|<2}
        \gamma^{2(2-|\alpha|)-1}
        \left|
        \bigl(\xi+\sqrt{-1}\gamma d\vartheta_q(x)\bigr)^\alpha
        \right|^2                                                     \\
        &\qquad\le C_H\left(
        |\mathfrak p_{q,\gamma}(x,\xi)|^2+
        \mathfrak b_q(x,\xi,\gamma)
        \right),
        \qquad \gamma\ge\gamma_1.
\end{split}
\end{equation}
This is precisely condition~(28.2.6) of
\cite[Theorem~28.2.3, p.~237]{HormanderIV} for the quadratic symbol
\(p_q\).

Let \(P_q^{(2)}\) denote the differential operator formed from the
second-order part of \(P_q\), so that its differential symbol is \(p_q\).
H\"ormander's Theorem~28.2.3, applied with \(X=X_1\) and with a
fixed relatively compact open set containing \(\operatorname{supp}\widetilde v\),
gives
\begin{equation}\label{eq:hormander-fixed-chart}
        \gamma^3
        \|e^{\gamma\vartheta_q}\widetilde v\|_{L^2(\mathcal K)}^2
        +
        \gamma
        \|e^{\gamma\vartheta_q}D\widetilde v\|_{L^2(\mathcal K)}^2
        \le
        C\|e^{\gamma\vartheta_q}P_q^{(2)}\widetilde v\|_{L^2(\mathcal K)}^2.
\end{equation}
The constant is uniform in \(q\).  To see this directly from
H\"ormander's proof, choose once and for all a cutoff supported in \(X_1\)
and equal to one on a fixed neighborhood of \(\mathcal K\).  The common estimate
\eqref{eq:hormander-286-uniform} supplies the same positivity constant in
(28.2.6).  After homogeneous normalization
\((\xi,\gamma)=\Lambda(\widehat\xi,\widehat\gamma)\), the localized symbols
and the weights form a bounded family in every symbol seminorm occurring
in the proof of Theorem~28.2.3.  In the proof of Theorem~28.2.3, the Weyl-symbol remainder is controlled in
the symbol class appearing there, and H\"ormander then applies the
Fefferman--Phong inequality \cite[Theorem~18.6.8]{HormanderIII} to the normalized Weyl
symbol.  The uniform symbol seminorm bounds above therefore make the
constants in both steps independent of \(q\).  The absorption in the proof
of Theorem~28.2.3 consequently gives common constants \(C\) and
\(\gamma_0\) for the entire family
\((P_q^{(2)},\vartheta_q)\).
Since
\[
        P_q=P_q^{(2)}+L_q
\]
with \(L_q\) a first-order operator whose coefficients are uniformly bounded
on \(\mathcal K\),
\[
        \|e^{\gamma\vartheta_q}L_q\widetilde v\|_{L^2(\mathcal K)}^2
        \le C\left(
        \|e^{\gamma\vartheta_q}D\widetilde v\|_{L^2(\mathcal K)}^2
        +\|e^{\gamma\vartheta_q}\widetilde v\|_{L^2(\mathcal K)}^2
        \right).
\]
The last expression is bounded by \(C\gamma^{-1}\) times the left-hand side
of \eqref{eq:hormander-fixed-chart}.  Increasing the common lower threshold
for \(\gamma\), we absorb this term and obtain
\begin{equation}\label{eq:fixed-chart-carleman}
        \gamma^3
        \|e^{\gamma\vartheta_q}\widetilde v\|_{L^2(\mathcal K)}^2
        +
        \gamma
        \|e^{\gamma\vartheta_q}D\widetilde v\|_{L^2(\mathcal K)}^2
        \le
        C\|e^{\gamma\vartheta_q}P_q\widetilde v\|_{L^2(\mathcal K)}^2.
\end{equation}

The factor \(e^{-\gamma\tau(q)}\) relating
\(e^{\gamma\vartheta_q}\) to \(e^{\gamma\tau\circ\Phi_q}\) occurs on both
sides and cancels.  Uniform equivalence of Euclidean and Riemannian norms and
measures in the normalized charts then yields
\begin{equation}\label{eq:uniform-local-carleman}
        \gamma^3
        \|e^{\gamma\tau}v\|_{L^2}^2
        +
        \gamma
        \|e^{\gamma\tau}\nabla_{\widehat g}v\|_{L^2}^2
        \le
        C
        \|e^{\gamma\tau}Pv\|_{L^2}^2
\end{equation}
for every \(\gamma\ge\gamma_0\), with a constant independent of \(q\) and
\(v\).

We next construct a uniformly locally finite covering. Reduce
$r_{\mathrm{loc}}$ so that balls of radius $10r_{\mathrm{loc}}$
about points of $E_S$ lie in the normalized charts above. Uniform
ellipticity and the fixed coordinate derivative bounds give a common
radius for geodesic normal coordinates, by the geodesic equation and
the inverse function theorem at the origin. Reducing
$r_{\mathrm{loc}}$ once more, if necessary, we therefore have
\[
 c\rho^{2n}\le
 \operatorname{Vol}_{\widehat g}(B_{\widehat g}(q,\rho))
 \le C\rho^{2n},
 \qquad 0<\rho\le10r_{\mathrm{loc}},\quad q\in E_S.
\]
Smooth radial cutoffs in these normal coordinates have uniform
derivative bounds. Choose a maximal
\(r_{\mathrm{loc}}/4\)-separated set
\[
        \{q_\nu\}_{\nu\in\mathcal I}\subset E_S
\]
with respect to the ambient metric \(\widehat g\).  Maximality gives
\[
        E_S
        \subset
        \bigcup_{\nu\in\mathcal I}
        B_{\widehat g}(q_\nu,r_{\mathrm{loc}}/2).
\]
The balls $B_{\widehat g}(q_\nu,r_{\mathrm{loc}}/8)$ are
pairwise disjoint. The preceding volume bounds therefore imply that
there is an integer $N_{\mathrm{ov}}$, independent of $\nu$, such
that every point belongs to at most $N_{\mathrm{ov}}$ of the enlarged
ambient balls
\[
        B_{\widehat g}(q_\nu,2r_{\mathrm{loc}}).
\]

Choose real-valued functions
\[
        \widetilde\psi_\nu
        \in
        C_c^\infty\bigl(B_{\widehat g}(q_\nu,r_{\mathrm{loc}})\bigr),
        \qquad
        0\le\widetilde\psi_\nu\le1,
\]
such that
\[
        \widetilde\psi_\nu\equiv1
        \quad\text{on}\quad
        B_{\widehat g}(q_\nu,r_{\mathrm{loc}}/2),
\]
and whose derivatives of every fixed finite order are uniformly bounded in the
normalized ACH charts.  On \(E_S\), set
\[
\Psi
:=
\left(
\sum_{\mu\in\mathcal I}
\widetilde\psi_\mu^2
\right)^{1/2},
\qquad
\psi_\nu
:=
\frac{\widetilde\psi_\nu}{\Psi}.
\]
Because the smaller balls cover \(E_S\), \(\Psi\ge1\).  Uniform finite
overlap and the uniform bounds for the \(\widetilde\psi_\nu\)'s give the
corresponding uniform bounds for \(\Psi\), \(\Psi^{-1}\), and the
\(\psi_\nu\)'s.  Moreover,
\begin{equation}\label{eq:quadratic-partition}
        \sum_{\nu\in\mathcal I}\psi_\nu^2=1.
\end{equation}
Thus \(\{\psi_\nu\}\) is a uniformly locally finite
\emph{quadratic partition of unity} subordinate to the ambient ACH covering.

Let \(w\in C_c^\infty(E_S)\).  Only finitely many \(\psi_\nu\)'s meet
\(\operatorname{supp}w\).  Apply
\eqref{eq:uniform-local-carleman} to \(\psi_\nu w\) and sum over \(\nu\).
By \eqref{eq:quadratic-partition},
\[
        \sum_\nu\|\psi_\nu w\|_{L^2}^2
        =
        \|w\|_{L^2}^2.
\]
Furthermore,
\[
        \sum_\nu\psi_\nu\,d\psi_\nu
        =
        \frac12
        d\left(\sum_\nu\psi_\nu^2\right)
        =
        0.
\]
Consequently, pointwise on \(E_S\),
\begin{align}
        \sum_\nu
        |\nabla_{\widehat g}(\psi_\nu w)|_{\widehat g}^2
        &=
        \sum_\nu
        \left|
        \psi_\nu\nabla_{\widehat g}w
        +
        w\nabla_{\widehat g}\psi_\nu
        \right|_{\widehat g}^2
        \notag\\
        &=
        |\nabla_{\widehat g}w|_{\widehat g}^2
        +
        |w|^2
        \sum_\nu
        |\nabla_{\widehat g}\psi_\nu|_{\widehat g}^2.
\label{eq:quadratic-partition-gradient-identity}
\end{align}
In particular,
\[
        \sum_\nu
        \|e^{\gamma\tau}
        \nabla_{\widehat g}(\psi_\nu w)\|_{L^2}^2
        \ge
        \|e^{\gamma\tau}
        \nabla_{\widehat g}w\|_{L^2}^2.
\]

On the right-hand side,
\[
        P(\psi_\nu w)
        =
        \psi_\nu Pw+[P,\psi_\nu]w.
\]
Since \(P=-\Delta_{\widehat g}\),
\[
        [P,\psi_\nu]w
        =
        -2\langle d\psi_\nu,dw\rangle_{\widehat g}
        -
        (\Delta_{\widehat g}\psi_\nu)w.
\]
Uniform derivative bounds and uniform finite overlap imply
\begin{equation}\label{eq:partition-commutator-sum}
        \sum_\nu
        \|e^{\gamma\tau}[P,\psi_\nu]w\|_{L^2}^2
        \le
        C
        \|e^{\gamma\tau}\nabla_{\widehat g}w\|_{L^2}^2
        +
        C
        \|e^{\gamma\tau}w\|_{L^2}^2.
\end{equation}
Also,
\[
        \sum_\nu
        \|e^{\gamma\tau}\psi_\nu Pw\|_{L^2}^2
        =
        \|e^{\gamma\tau}Pw\|_{L^2}^2.
\]
Therefore, after summing the local estimates,
\[
\begin{aligned}
        \gamma^3
        \|e^{\gamma\tau}w\|_{L^2}^2
        +
        \gamma
        \|e^{\gamma\tau}\nabla_{\widehat g}w\|_{L^2}^2
        \le{}&
        C
        \|e^{\gamma\tau}Pw\|_{L^2}^2                       \\
        &+
        C
        \|e^{\gamma\tau}\nabla_{\widehat g}w\|_{L^2}^2
        +
        C
        \|e^{\gamma\tau}w\|_{L^2}^2.
\end{aligned}
\]
After increasing \(\gamma_0\), the final two terms are absorbed into the
left-hand side.  This proves \eqref{eq:graph-carleman}.
\end{proof}

\begin{lemma}
\label{lem:graph-carleman-extension}
Let \(w\in C^\infty(E_S)\) have support contained in
\[
        \{\tau_-\le\tau\le\tau_+\}
\]
for some
\[
        0<\tau_-<\tau_+<\infty.
\]
Suppose, for a fixed \(\gamma\ge\gamma_0\), that
\[
        e^{\gamma\tau}w,\qquad
        e^{\gamma\tau}\nabla_{\widehat g}w,\qquad
        e^{\gamma\tau}Pw
        \in L^2(E_S,dV_{\widehat g}).
\]
Then \eqref{eq:graph-carleman} holds for \(w\).
\end{lemma}

\begin{proof}
Choose \(\chi\in C^\infty(\mathbb R)\) satisfying
\[
        0\le\chi\le1,
        \qquad
        \chi(t)=0\text{ for }t\le1,
        \qquad
        \chi(t)=1\text{ for }t\ge2.
\]
Let \(\varepsilon_j\downarrow0\) and put
\[
        \zeta_j(y)
        =
        \chi\left(\frac{\delta(y)}{\varepsilon_j}\right).
\]
Then
\[
        \zeta_jw\in C_c^\infty(E_S).
\]
Indeed, on the support of \(\zeta_jw\) one has
\[
        \delta\ge\varepsilon_j,
        \qquad
        \tau_-\le\tau\le\tau_+,
\]
and hence
\[
        S-\log\delta+\tau_-
        \le s
        \le
        S-\log\varepsilon_j+\tau_+.
\]

On the support of \(d\zeta_j\), we have
\[
        \varepsilon_j\le\delta\le2\varepsilon_j.
\]
The ordinary first and second derivatives of \(\zeta_j\) are respectively
\(O(\varepsilon_j^{-1})\) and \(O(\varepsilon_j^{-2})\).  On \(E_S\),
however,
\[
        e^{-s}\le e^{-S}\delta.
\]
Consequently the ACH-scaled first and second derivatives are uniformly
bounded.  In particular,
\begin{equation}\label{eq:zeta-uniform-bounds}
        |d\zeta_j|_{\widehat g}
        +
        |\Delta_{\widehat g}\zeta_j|
        \le C
\end{equation}
with \(C\) independent of \(j\).

Applying Proposition~\ref{prop:theta-carleman} to \(\zeta_jw\) gives
\[
\begin{aligned}
        \gamma^3\|e^{\gamma\tau}\zeta_jw\|^2
        +
        \gamma
        \|e^{\gamma\tau}
        \nabla_{\widehat g}(\zeta_jw)\|^2
        \le
        C\|e^{\gamma\tau}P(\zeta_jw)\|^2.
\end{aligned}
\]
Since \(P=-\Delta_{\widehat g}\),
\[
        P(\zeta_jw)
        =
        \zeta_jPw
        -
        2\langle d\zeta_j,dw\rangle_{\widehat g}
        -
        (\Delta_{\widehat g}\zeta_j)w.
\]
By \eqref{eq:zeta-uniform-bounds},
\[
        |[P,\zeta_j]w|
        \le
        C\mathbf 1_{\{\varepsilon_j\le\delta\le2\varepsilon_j\}}
        \left(
        |\nabla_{\widehat g}w|+|w|
        \right).
\]
The sets
\[
        \{\varepsilon_j\le\delta\le2\varepsilon_j\}
        \cap\{\tau_-\le\tau\le\tau_+\}
\]
leave every compact subset of \(E_S\): indeed, \(\delta\) has a positive
minimum on each compact subset of \(E_S\).  The assumed weighted
\(L^2\)-integrability and absolute continuity of the integral therefore give
\[
        e^{\gamma\tau}[P,\zeta_j]w\longrightarrow0
        \qquad\text{in }L^2.
\]
Likewise,
\[
\begin{split}
        e^{\gamma\tau}\zeta_jw
        &\longrightarrow e^{\gamma\tau}w,\\
        e^{\gamma\tau}\nabla_{\widehat g}(\zeta_jw)
        &\longrightarrow
        e^{\gamma\tau}\nabla_{\widehat g}w,\\
        e^{\gamma\tau}\zeta_jPw
        &\longrightarrow e^{\gamma\tau}Pw
\end{split}
\]
in the corresponding \(L^2\)-spaces.  Letting \(j\to\infty\) proves the
lemma.
\end{proof}

\begin{lemma}
\label{lem:flatness-weighted-integrability}
Let \(u\in\dot C^\infty(X_\Theta)\) over a neighborhood of
\(\overline Y\).  Then, for every \(M>0\) and \(j=0,1,2\),
\begin{equation}\label{eq:flatness-s-estimate}
        |\nabla_{\widehat g}^ju|
        \le C_{M,j}e^{-Ms}
\end{equation}
on \(E_S\).  Consequently, for every \(\gamma>0\) and every compact interval
\(I\Subset(0,\infty)\),
\begin{equation}\label{eq:weighted-strip-integrability}
        \mathbf 1_{\{\tau\in I\}}
        e^{\gamma\tau}\nabla_{\widehat g}^ju
        \in L^2(E_S,dV_{\widehat g}),
        \qquad j=0,1,2.
\end{equation}
Furthermore, if \(M>n+\gamma\), then for every \(R\ge0\),
\begin{equation}\label{eq:weighted-tail-decay}
        \left\|
        \mathbf 1_{\{\tau\ge R\}}
        e^{\gamma\tau}\nabla_{\widehat g}^ju
        \right\|_{L^2}
        \le C_{M,\gamma,j}e^{-(M-n-\gamma)R},
        \qquad j=0,1,2.
\end{equation}
\end{lemma}

\begin{proof}
Because \(\widetilde x/x\) extends positively to the compactified boundary and
both \(\widetilde x/x\) and \(x/\widetilde x\) are bounded, infinite-order
vanishing in \(x\) is equivalent to infinite-order vanishing in
\(\widetilde x=e^{-s}\).  The bounded derivative comparison between the
corresponding normalized \(\Theta\)-frames, through order two, gives the same
equivalence for the derivatives used below. Covariant derivatives
in the \(\widehat g\)-bounded frame are combinations, with uniformly bounded
coefficients, of the \(\Theta\)-derivatives
\[
        \widetilde x\partial_{\widetilde x},
        \qquad
        \widetilde xY_j,
        \qquad
        \widetilde x^2T.
\]
Thus \eqref{eq:flatness-s-estimate} follows from
\(u\in\dot C^\infty(X_\Theta)\).

Write
\[
        s=S-\log\delta(y)+\tau
\]
on \(E_S\).  By \eqref{eq:ach-volume-comparison}, for
\(I=[A,B]\Subset(0,\infty)\),
\[
\begin{aligned}
 &\int_{\{\tau\in I\}}
 e^{2\gamma\tau}
 |\nabla_{\widehat g}^ju|^2\,dV_{\widehat g}               \\
 &\qquad\le
 C\int_Y\int_A^B
 e^{2\gamma\tau}
 e^{-2(M-n)(S-\log\delta+\tau)}
 \,d\tau\,d\mu_0(y)                                        \\
 &\qquad\le
 C_{A,B,M,\gamma}
 \int_Y\delta(y)^{2(M-n)}\,d\mu_0(y).
\end{aligned}
\]
The last integral is finite once \(M>n\), proving
\eqref{eq:weighted-strip-integrability}.

Integrating the same estimate over \(\tau\ge R\) gives, when
\(M>n+\gamma\),
\[
\begin{aligned}
 &\int_{\{\tau\ge R\}}
 e^{2\gamma\tau}
 |\nabla_{\widehat g}^ju|^2\,dV_{\widehat g}                 \\
 &\qquad\le
 C\int_Y\delta(y)^{2(M-n)}\,d\mu_0(y)
 \int_R^\infty e^{-2(M-n-\gamma)\tau}\,d\tau                \\
 &\qquad\le
 C_{M,\gamma,j}e^{-2(M-n-\gamma)R}.
\end{aligned}
\]
Taking square roots proves \eqref{eq:weighted-tail-decay}.
\end{proof}

With the conventions above,
\(\Delta_{g_B^{\mathrm R}}=2\Delta_{\gB}\) on scalar functions.
Since \(\widehat g=c_Bg_B^{\mathrm R}\), it follows that
\begin{equation}\label{eq:real-complex-laplacian}
        \Delta_{\widehat g}
        =c_{\mathrm R}\Delta_{\gB},
        \qquad c_{\mathrm R}:=\frac2{c_B}=4(n+1).
\end{equation}
Consequently,
\[
        P=-\Delta_{\widehat g}
        =
        -c_{\mathrm R}\Delta_{\gB},
\]
and hence
\[
        \Delta_{\gB}u=0
        \quad\Longrightarrow\quad
        Pu=0.
\]

\begin{proposition}[Local unique continuation from an ACH boundary patch]
\label{prop:uci}
Let \(\Gamma_1\Subset M_0\), and let \(u\) be smooth on a one-sided
collar of \(\Gamma_1\).  Assume that
\[
        \Delta_{\gB}u=0
\]
on this collar and that
\[
        u\in\dot C^\infty(X_\Theta)
\]
over \(\Gamma_1\).  Then there are a nonempty open set
\(Y_0\Subset\Gamma_1\) and \(S_2>0\) such that
\[
        u=0
        \qquad\text{on}\qquad
        \{s>S_2\}\times Y_0.
\]
\end{proposition}

\begin{proof}
Choose
\[
        \overline Y\Subset \Gamma_0 \Subset\Gamma_1
\]
as above and take \(S\) sufficiently large that
Lemmas~\ref{lem:convex-graph-end} and~\ref{lem:flatness-weighted-integrability} apply.  Increase \(S\), if
necessary, so that \(E_S\) lies in the collar on which
\(\Delta_{\gB}u=0\).  Set
\[
        P=-\Delta_{\widehat g}.
\]
By \eqref{eq:real-complex-laplacian},
\[
        Pu=0
        \qquad\text{on }E_S.
\]

Fix \(\tau_1>2\).  Choose
\(\chi_-\in C^\infty(\mathbb R)\) satisfying
\[
        \chi_-(t)=0\text{ for }t\le\tau_1,
        \qquad
        \chi_-(t)=1\text{ for }t\ge\tau_1+1.
\]
Choose once and for all a function
\(\chi_+\in C^\infty(\mathbb R)\) such that
\[
        0\le\chi_+\le1,
        \qquad
        \chi_+(t)=1\text{ for }t\le0,
        \qquad
        \chi_+(t)=0\text{ for }t\ge1.
\]
For \(R>\tau_1+3\), set
\[
        \chi_R(t):=\chi_+(t-R).
\]
Then
\[
        \chi_R(t)=1\text{ for }t\le R,
        \qquad
        \chi_R(t)=0\text{ for }t\ge R+1,
\]
and, for every fixed \(k\),
\[
        \sup_R\|\chi_R^{(k)}\|_{L^\infty}<\infty.
\]
Put
\[
        v_R=\chi_-(\tau)\chi_R(\tau)u.
\]
The function \(v_R\) has compact support in the \(\tau\)-variable.  By
Lemma~\ref{lem:flatness-weighted-integrability}, for every fixed
\(\gamma\ge\gamma_0\),
\[
        e^{\gamma\tau}v_R,\qquad
        e^{\gamma\tau}\nabla_{\widehat g}v_R,\qquad
        e^{\gamma\tau}Pv_R
        \in L^2(E_S,dV_{\widehat g}).
\]
Lemma~\ref{lem:graph-carleman-extension} therefore gives
\begin{equation}\label{eq:carleman-vR}
\begin{split}
        \gamma^3\|e^{\gamma\tau}v_R\|^2
        +
        \gamma
        \|e^{\gamma\tau}\nabla_{\widehat g}v_R\|^2
        \le
        C\|e^{\gamma\tau}Pv_R\|^2.
\end{split}
\end{equation}

Set
\[
        \alpha=\chi_-(\tau),
        \qquad
        \beta_R=\chi_R(\tau).
\]
Because \(R>\tau_1+3\), we have
\[
\beta_R\equiv1
\quad\text{on }\operatorname{supp}d\alpha,
\qquad
\alpha\equiv1
\quad\text{on }\operatorname{supp}d\beta_R,
\]
and
\[
\langle d\alpha,d\beta_R\rangle_{\widehat g}=0.
\]
For \(P=-\Delta_{\widehat g}\), the product rule gives
\[
P(\alpha\beta_Ru)
=
\alpha\beta_RPu
+
\beta_R[P,\alpha]u
+
\alpha[P,\beta_R]u
-
2\langle d\alpha,d\beta_R\rangle_{\widehat g}u.
\]
Since \(Pu=0\), the preceding support relations imply
\[
        Pv_R
        =
        [P,\chi_-(\tau)]u
        +
        [P,\chi_R(\tau)]u.
\]
The second term is supported where
\[
        R\le\tau\le R+1.
\]
The uniform derivative bounds for \(\chi_R\), together with the uniform
bounds for \(d\tau\) and \(\nabla^2\tau\), give
\[
        |[P,\chi_R(\tau)]u|
        \le
        C\mathbf 1_{\{R\le\tau\le R+1\}}
        \bigl(|\nabla_{\widehat g}u|+|u|\bigr),
\]
with \(C\) independent of \(R\).  Hence
Lemma~\ref{lem:flatness-weighted-integrability}, in particular the weighted
tail estimate \eqref{eq:weighted-tail-decay}, gives for every fixed
\(\gamma\)
\[
        \|e^{\gamma\tau}[P,\chi_R(\tau)]u\|_{L^2}
        \longrightarrow0
        \qquad
        \text{as }R\to\infty.
\]
The full-tail estimate \eqref{eq:weighted-tail-decay} also gives
\[
        v_R\longrightarrow\chi_-(\tau)u
\]
in the weighted \(H^1\)-norm for each fixed \(\gamma\).  Letting
\(R\to\infty\) in \eqref{eq:carleman-vR}, we obtain
\begin{equation}\label{eq:carleman-fixed-strip}
        \gamma^{3/2}
        \|e^{\gamma\tau}\chi_-(\tau)u\|_{L^2}
        \le
        C
        \|e^{\gamma\tau}[P,\chi_-(\tau)]u\|_{L^2}.
\end{equation}
The commutator on the right is supported in the fixed strip
\[
        \tau_1\le\tau\le\tau_1+1.
\]
Its unweighted \(L^2\)-norm is finite by
Lemma~\ref{lem:flatness-weighted-integrability}.  Thus there is a constant
\(C_{\tau_1}>0\), independent of \(\gamma\), such that
\begin{equation}\label{eq:fixed-strip-upper-bound}
        \gamma^{3/2}
        \|e^{\gamma\tau}\chi_-(\tau)u\|_{L^2}
        \le
        C_{\tau_1}e^{\gamma(\tau_1+1)}.
\end{equation}

Suppose that \(u(q)\neq0\) for some \(q\in E_S\) with
\[
        \tau(q)>\tau_1+2.
\]
By continuity, there are an open set \(V\Subset E_S\), a number
\(a>\tau_1+2\), and a constant \(c_V>0\) such that
\[
        \tau\ge a,
        \qquad
        |u|\ge c_V
        \qquad\text{on }V.
\]
Since \(\chi_-=1\) on \(V\),
\[
        \|e^{\gamma\tau}\chi_-u\|_{L^2}
        \ge
        c_V\operatorname{Vol}_{\widehat g}(V)^{1/2}
        e^{\gamma a}.
\]
This contradicts \eqref{eq:fixed-strip-upper-bound} as
\(\gamma\to\infty\), because
\[
        a>\tau_1+1.
\]
Therefore
\[
        u=0
        \qquad\text{on}\qquad
        \{\tau>\tau_1+2\}\cap E_S.
\]

Finally choose any nonempty open set \(Y_0\Subset Y\).  Since
\(\varphi\) is bounded on \(\overline Y_0\), the number
\[
        S_2
        :=
        S+\max_{\overline Y_0}\varphi+\tau_1+2
\]
is finite, and
\[
        \{s>S_2\}\times Y_0
        \subset
        \{\tau>\tau_1+2\}.
\]
Hence
\[
        u=0
        \qquad\text{on}\qquad
        \{s>S_2\}\times Y_0,
\]
as required.
\end{proof}

\begin{proof}[Proof of Theorem~\ref{thm:csck-ke}]
Set $u=\log(\cB/\Cn)$. By Lemma~\ref{lem:scalar-pde} and
Proposition~\ref{prop:infinite-vanishing}, $u$ is Bergman-harmonic and
flat at the strongly pseudoconvex patch. Proposition~\ref{prop:uci}
therefore gives $u=0$ on a nonempty one-sided collar. Since $u$ is
real analytic on the connected domain $\om$, it vanishes identically.
Thus $\cB=\Cn$, and
\[
        \Ric(\gB)=-\gB-\ddbar\log\cB=-\gB.
\]
\end{proof}

\section{Extension for possibly unbounded domains}
\label{sec:unbounded-extension}

We now remove the boundedness assumption by allowing the
Bergman kernel to have a zero set in the interior.  We shall use the
following local nondegeneracy result, proved in
\cite{EGSXscalarcurv2026}.

\begin{proposition}
\label{prop:unbounded-finite-type-nondegeneracy}
Let $\om\subset\C^n$ be a possibly unbounded pseudoconvex domain, and let
$p\in\partial\om$ be a smooth boundary point of finite D'Angelo type.
Then there exists a neighborhood $V$ of $p$ such that
\[
        K_\om(z,z)>0,
        \qquad
        \ddbar\log K_\om(z,z)>0
        \qquad\text{on }V\cap\om.
\]
In particular, the Bergman metric is well-defined on $V\cap\om$.
\end{proposition}

\begin{proof}[Proof of Theorem~\ref{thm:unbounded-extension}] Choose a smooth strongly pseudoconvex open patch $M\subset\partial\Omega$ containing $p_0$, and a relatively compact subpatch $p_0\in M_0\subset M$.
Since every point of the strongly pseudoconvex patch $M$ has finite
D'Angelo type, Proposition~\ref{prop:unbounded-finite-type-nondegeneracy}
shows that the Bergman metric is well-defined on a one-sided neighborhood
of each point of $M$.  Shrinking $M_0$ if necessary,
we fix a connected one-sided collar on which
\[
        K_\om(z,z)>0,
        \qquad
        \gB>0.
\]
The localization theorem of Hsiao--Huang--Li \cite[Theorem~1.3]{HHL2026} applies to possibly
unbounded pseudoconvex domains. Near the strongly pseudoconvex patch,
it gives a bounded smoothly bounded strongly pseudoconvex domain $D$
with the same boundary germ and with $K_\om-K_D$ smooth up to that
patch. Hence $K_\om$ has the local Fefferman expansion used above.
Together with local nondegeneracy, this is all that is needed for
Propositions~\ref{prop:local-boundary-asymptotics},
\ref{prop:infinite-vanishing}, and \ref{prop:uci}. Their differential
equations and Carleman estimates concern $g_B$ on the collar;
boundedness is used only for the auxiliary domain $D$.

Suppose first that $\Scal(\gB)$ is constant on the collar.
Lemma~\ref{lem:scalar-pde} gives
\[
        \Scal(\gB)=-n,
        \qquad
        \Delta_{\gB}u=0,
        \qquad
        u:=\log\frac{\cB}{\Cn}.
\]
Proposition~\ref{prop:infinite-vanishing} and
Proposition~\ref{prop:uci} then imply, after shrinking to a nonempty
subcollar, that
\[
        u=0.
\]
Hence
\begin{equation}\label{eq:unbounded-local-B-constant}
        \cB=\Cn
\end{equation}
on a nonempty open subset of $\om$.

Now suppose instead that $\gB$ is extremal on the collar.
By Corollary~\ref{cor:X-zero-collar}, after shrinking the collar,
\[
        \mathcal X
        :=
        \nabla_{\gB}^{1,0}\Scal(\gB)
        =0.
\]
Since $\gB$ is positive definite there,
\[
        \bar\partial\Scal(\gB)=0.
\]
The scalar curvature is real-valued, so it is constant on the connected
collar.  Proposition~\ref{prop:local-boundary-asymptotics} then shows that
this constant is $-n$.  We are therefore reduced to the preceding cscK
case and again obtain \eqref{eq:unbounded-local-B-constant} on a nonempty
open subset.

It remains to propagate this local identity without assuming that
$K_\om$ is everywhere positive.  Write
\[
        K:=K_\om(z,z)
\]
on $\om$, and define the Hermitian matrix
\[
        H_{\alpha\bar\beta}
        :=
        K K_{\alpha\bar\beta}
        -
        K_\alpha K_{\bar\beta}.
\]
Where $K>0$,
\[
        (\gB)_{\alpha\bar\beta}
        =
        \frac{H_{\alpha\bar\beta}}{K^2}.
\]
Consequently,
\[
        G
        =
        \det\bigl((\gB)_{\alpha\bar\beta}\bigr)
        =
        \frac{\det(H_{\alpha\bar\beta})}{K^{2n}},
\]
and hence
\begin{equation}\label{eq:unbounded-cross-multiplied}
        \cB=\Cn
        \quad\Longleftrightarrow\quad
        \det(H_{\alpha\bar\beta})
        =
        \Cn K^{2n+1}.
\end{equation}

The Bergman kernel is holomorphic in its first variable and
antiholomorphic in its second variable.  Thus its diagonal restriction
$K_\om(z,z)$ is real analytic on $\om$, and so is
\[
        F
        :=
        \det(H_{\alpha\bar\beta})
        -
        \Cn K^{2n+1}.
\]
By \eqref{eq:unbounded-local-B-constant} and
\eqref{eq:unbounded-cross-multiplied}, $F$ vanishes on a nonempty open
subset of the connected domain $\om$.  The identity theorem for
real-analytic functions therefore gives
\begin{equation}\label{eq:unbounded-global-cross-identity}
        \det(H_{\alpha\bar\beta})
        =
        \Cn K^{2n+1}
        \qquad\text{on all of }\om.
\end{equation}

We next show that the Bergman form is positive definite everywhere on
$\om_K=\{K>0\}$.  Let $\{\varphi_j\}$ be an orthonormal basis of
$A^2(\om)$. Local $L^2$ mean-value estimates and Cauchy estimates
give locally normal convergence of the kernel expansion and its
derivatives. For $z\in\om_K$ and
$X\in T_z^{1,0}\om$, the standard Bergman-kernel calculation gives
\[
\begin{aligned}
(\partial\bar\partial\log K)(X,\overline X)
&=
\frac{
K\sum_j|X\varphi_j|^2
-
\left|\sum_j(X\varphi_j)\overline{\varphi_j}\right|^2
}{K^2}\\
&\ge0
\end{aligned}
\]
by the Cauchy--Schwarz inequality.  Hence $\gB$ is positive semidefinite
on $\om_K$.

On the other hand, \eqref{eq:unbounded-global-cross-identity} yields,
for $z\in\om_K$,
\[
        G
        =
        \frac{\det(H_{\alpha\bar\beta})}{K^{2n}}
        =
        \Cn K
        >0.
\]
A positive-semidefinite Hermitian matrix with positive determinant is
positive definite.  Thus $\gB$ is well-defined on all of $\om_K$.
Moreover,
\[
        \cB=\frac{G}{K}=\Cn
        \qquad\text{on }\om_K.
\]
Using
\[
        \Ric(\gB)
        =
        -\gB-\ddbar\log\cB,
\]
we conclude that
\[
        \Ric(\gB)=-\gB
        \qquad\text{on }\om_K.
\]

If $K_\om(z,z)>0$ everywhere on $\om$, then
$\om_K=\om$, and the final assertion follows.
\end{proof}

\bibliographystyle{amsplain}
%\bibliography{ref_asympcsck}

\providecommand{\bysame}{\leavevmode\hbox to3em{\hrulefill}\thinspace}
\providecommand{\MR}{\relax\ifhmode\unskip\space\fi MR }
% \MRhref is called by the amsart/book/proc definition of \MR.
\providecommand{\MRhref}[2]{%
  \href{http://www.ams.org/mathscinet-getitem?mr=#1}{#2}
}
\providecommand{\href}[2]{#2}

\end{document}